\documentclass[]{article}

\usepackage{amsmath,amsthm,amssymb,amstext,graphicx}
\usepackage{subfigure}
\usepackage{fancyvrb}
\usepackage{enumerate}
\usepackage{a4wide}

\newtheorem{theorem}{Theorem}[section]
\newtheorem{corollary}[theorem]{Corollary}
\newtheorem{lemma}[theorem]{Lemma}
\newtheorem{proposition}[theorem]{Proposition}

\theoremstyle{definition}
\newtheorem{algorithm}{Algorithm}
\theoremstyle{definition}
\newtheorem{definition}[theorem]{Definition}
\newtheorem{remark}[theorem]{Remark}

\newcommand{\ep}{\varepsilon}

\newcommand{\R}{\mathbb{R}}

\newcommand{\N}{\mathbb{N}}

\renewcommand{\P}{\mathcal{P}}

\newcommand{\T}{\mathcal{T}}
\newcommand{\A}{\mathcal{A}}
\newcommand{\D}{\mathcal{D}}

\newcommand{\dvg}{\mathop{\mathrm{div}}}
\newcommand{\intd}{\ \mathrm{d}}

\title{Estimating long term behavior of flows without trajectory integration: the infinitesimal generator approach}

\author{Gary Froyland\thanks{School of Mathematics and Statistics, University of New South Wales, Sydney, NSW 2052, Australia.}
        \and Oliver Junge\thanks{Faculty of Mathematics, Technische Universit\"at M\"unchen, 85748 Garching, Germany.}
        \and P{\'e}ter Koltai\thanks{Faculty of Mathematics, Technische Universit\"at M\"unchen, 85748 Garching, Germany. P\'eter Koltai was partly supported by the TopMath PhD program within the Elite Network of Bavaria and the TUM Graduate School.}
        }

\begin{document}

\maketitle

\begin{abstract}
The long-term distributions of trajectories of a flow are described by invariant densities, i.e.\ fixed points of an associated transfer operator.  In addition, global slowly mixing structures, such as almost-invariant sets, which partition phase space into regions that are almost dynamically disconnected, can also be identified by certain eigenfunctions of this operator.  Indeed, these structures are often hard to obtain by brute-force trajectory-based analyses.  In a wide variety of applications, transfer operators have proven to be very efficient tools for an analysis of the global behavior of a dynamical system.  

The computationally most expensive step in the construction of an approximate transfer operator is the numerical integration of many short term trajectories.  In this paper, we propose to directly work with the \emph{infinitesimal generator} instead of the operator, completely avoiding trajectory integration.  We propose two different discretization schemes; a cell based discretization and a spectral collocation approach.  Convergence can be shown in certain circumstances.  We demonstrate numerically that our approach is much more efficient than the operator approach, sometimes by several orders of magnitude.
\end{abstract}

%\pagestyle{myheadings}
%\thispagestyle{plain}
%\markboth{G. FROYLAND, O. JUNGE AND P. KOLTAI}{\textsc{THE INFINITESIMAL GENERATOR APPROACH}}

\section{Introduction}
Analysis of the long-term behavior of flows can be broadly classified into geometric methods and statistical methods.
Geometrical methods include the determination of fixed points, periodic orbits, and invariant manifolds.  Invariant manifolds of fixed points or periodic orbits act as barriers to transport as  trajectories may not cross the manifolds transversally. Statistical methods include determining the distribution of points in very long trajectories of a very large set of initial points (i.e.\ a physical invariant measure \cite{Yo02a}, often possessing an invariant density) and the identification of meta-stable or almost-invariant sets \cite{DeJu99a,FrDe03,Fr05}.  Almost-invariant sets partition the phase space into almost dynamically disconnected regions and are important for revealing global dynamical structures that are often invisible to an analysis of trajectories. These metastable dynamics also go under the names of persistent patterns, or strange eigenmodes, both of which are realisable as eigenfunctions of a transfer operator. Frequently, the boundaries of maximally almost-invariant sets (those sets which are locally closest to invariant sets) coincide with certain invariant manifolds \cite{FrPa09a}.

Our focus in the present paper is on statistical methods, although we demonstrate via case studies the relationships to geometrical methods.
The most commonly used tool for statistical methods is the transfer operator (or Perron-Frobenius operator).
Fixed points of the transfer operator correspond to invariant densities, while eigenfunctions corresponding to real positive eigenvalues strictly less than one provide information on almost-invariant sets.
In practice, one typically constructs a finite-rank numerical approximation of a transfer operator and computes large spectral values and eigenfunctions for this finite-rank operator.
The construction of the finite-rank approximation requires the integration of many relatively short trajectories with initial points sampled over the domain of the flow.
It is this use of {\em short} trajectories that gives the transfer operator approach additional stability and accuracy when compared with computations based upon very long trajectories.
Long trajectories continually accumulate small errors from imperfect numerical integration and finite computer representation of numbers;  these small errors quickly grow in chaotic flows.
While the transfer operator approach is very stable, it still requires the computation of many small trajectories which can be very time consuming in some systems.
The approach we describe in the present work obviates the need for any trajectory integration at all and works directly with the vector field.

Our approach exploits the fact that the evolution of probability densities $u=u(t)$ can be described by generalized solutions of the abstract Cauchy problem ${\partial_t u = \A u}$. The Perron--Frobenius operator is the evolution operator of this equation, and has the same eigenfunctions as the operator~$\A$. The operator $\A$ is an unbounded hyperbolic (if the underlying dynamics is deterministic) or elliptic (if the deterministic dynamics is perturbed by white noise) partial differential operator. Standard techniques allow us to approximate the eigenmodes we are interested in: finite difference, finite volume, finite element and spectral methods yield such discretizations; see~\cite{QuVa94,Kron97,Lev02,Boyd01,CaHuQu91a} and the references therein.

An outline of the paper is as follows.
In Section 2 we provide background on the infinitesimal operator arising from smooth vector fields and describe conditions under which the operator generates a semigroup of transfer operators. In order to obtain formal results, we will require the addition of a small amount of diffusion to the deterministic flow. In the small diffusion setting we discuss existence of invariant densities and spectral results for the associated infinitesimal operator.
Section 3 describes a spectral Galerkin method for the approximation of the infinitesimal operator.  We apply results from the numerical analysis of advection-diffusion PDEs to show that our Galerkin method approximates the true eigenfunctions of the infinitesimal operator as our Galerkin basis becomes increasingly refined.  We can also show that the convergence rate is spectral; that is, faster than any polynomial.
Section 4 describes an Ulam-based Galerkin method for approximating the infinitesimal operator. This Ulam-based approach is new and shares some similarities with finite-difference schemes. While we cannot show convergence of this approximation scheme, the numerical results obtained are extremely fast and accurate.
Sections 5--8 detail the practical application of our two approximation methods to flows in one-, two-, and three- dimensional domains.  We demonstrate the spectral accuracy of our spectral Galerkin method and compare with the accuracy of the Ulam-based Galerkin method.
We also compare the accuracy vs.\ computational effort of our two new approaches with standard transfer operator approaches. The natural relationships between the outputs of our statistical methods and geometric objects such as the vector field and invariant manifolds are also elucidated in each case study.
We conclude in Section 9.

%===================
%SECTION: BACKGROUND
%===================

%\section{Background}    \label{sec:bg}

\section{Dynamics, densities and semigroups}

Let the domain $M\subset\R^d$ of our flow be a smooth  compact
manifold and $m$ the (normalized) Lebesgue measure on $M$. Denote
by $F:M\to\R^d$ the vector field generating the flow and by $\Phi^t:M\to M$, $t\in\R$, the flow, i.e.\ $\Phi^t(x)$ represents the location of a trajectory beginning at $x\in M$ after flowing for $t\in\R$ time units. One has that $d\Phi^t(x)/dt=F(x)$. Note that -- provided that the components $F_i, i = 1,\ldots,d$ have continuous derivatives -- the function $\Phi^t:M\to M$ is a diffeomorphism for every $t\in\R$.

Invariant sets are structures of dynamical interest.  A set $A\subset M$
is called \emph{invariant}, iff $\Phi^{-t}A = A$ for all $t$.  
Also, one asks how the flow changes probability measures. Sample
$x$ according to a probability measure $\mu$; the distribution of
$\Phi^tx$ is then given by $\mu\circ\Phi^{-t}$. Special attention is
to be drawn to \textit{invariant measures}, which do not change
under the dynamics ($\mu = \mu\circ\Phi^{-t}$). Invariant measures
$\mu$ are called \textit{ergodic} if invariant sets have either zero
or full measure, i.e.\ if $A\subset M$ satisfies $\Phi^{-t}A = A$
then $\mu(A)\in\{0,1\}$.
% OJ: removed, I think this is too dense at this point
% That means there is no nontrivial decomposition of the dynamical system in terms of $\mu$.
An even more restricted class of ergodic measures are the \textit{physically relevant} (or \textit{natural}) ones, satisfying
\[ \int\psi\intd\mu = \lim_{T\to\infty}\frac{1}{T}\int_0^T\psi\left(\Phi^tx\right)\intd t \]
for all $\psi:M\to\R$ continuous and $x\in U\subset M$ with
$m(U)>0$. One has that absolutely continuous\footnote{$\mu$ is
absolutely continuous (with respect to $m$, if not specified
further), if there is an $0\leq f\in L^1(m)$ such that for all
$m$-measurable $B\subset M$ we have $\mu(B)=\int_B f\intd m$. The
function $f$ is called the density of $\mu$.} ergodic measures are
natural invariant measures. The density of an invariant measure is
called the \textit{invariant density}.

\subsection{Transfer operators}
\label{subsec.transfer}

When looking for invariant densities, one can rephrase the action of
the flow on measures as an action on densities. If $f$ denotes the
density of $\mu$, then $f$ is evolved by the flow as
\[
\P^tf\ (x) = f\left(\Phi^{-t}x\right)|D\Phi^{-t}x|,
\]
where $|B|$ denotes $|\det B|$ for a matrix $B$. The linear operator
$\P^t:L^1(M)\circlearrowleft$ is known as a \textit{transfer
operator} or the \textit{Perron-Frobenius operator} associated with
the flow $\Phi$. Note that invariant densities are fixed points of
$\P^t$.
For each $t\in\R$ we have
\begin{enumerate}
 \item[(i)] $\P^t$ is linear,
 \item[(ii)] $\P^tf\geq 0$ if $f\geq0$ and,
 \item[(iii)] $\|\P^tf\|=\|f\|$ for all $f\geq0$, where $\|\cdot\|$ is
 the $L^1$ norm.
\end{enumerate}
i.e.\ $\P^t$ is a \emph{Markov operator}. Moreover, these properties
also imply $\|\P^t\|\leq 1$, hence $\P^t$ is a contraction and its
eigenvalues lie inside the unit disc. In fact, since the flow $\Phi^t$ is a
bijection, we have $\|\P^tf\|=\|f\|$ for all $f\in L^1$. To see
this, note that
\begin{equation}\label{eq:bijection}
\int_A\P^tf \; dm = \int_{\Phi^{-t}A}f\; dm
\end{equation}
for $t>0$ and $A\subset M$ measurable. Now write $f=f^+-f^-$, with
$f^\pm(x) = \max\{\pm f(x),0\}$, then $\|f\|=\int (f^+ - f^-) = \int_{S^+} f^+ + \int_{S^-} f^-$, where $S^\pm$ is the support of $f^\pm$. Now use properties (ii), (iii) and (\ref{eq:bijection}) with $A=S^\pm$.  A consequence of this is that the eigenvalues of $\P^t$ lie on the unit circle.

\subsection{Operator semigroups, generators}

The transfer operator also inherits some (semi)group properties of the flow $\Phi^t$.
\begin{definition}
 Let $(X,\|\cdot\|)$ be a Banach space. A one parameter family $\left\{\T^t\right\}_{t\geq0}$ of bounded linear operators $\T^t:X\to X$ is called a \emph{semigroup} on $X$, if
 \begin{itemize}
     \item[(a)] $\T^0=I$ ($I$ denoting the identity on $X$),
     \item[(b)] $\T^{t+s} = \T^t\T^s$ for all $t,s\geq 0$.
 \end{itemize}
 Further, if $\|\T^t\|\leq1$, the family is called a semigroup of contractions. 
 
 If
$\lim_{t\to 0}\|\T^tf- f\|=0$ for every $f\in X$, then $\T^t$ is a \emph{continuous semigroup} (\emph{$C_0$ semigroup}).
\end{definition}
The transfer operator $\P^t$ is a $C_0$ semigroup of
contractions on $L^1$, see \cite{LaMa94a} for a proof (in particular
Remark 7.6.2 for the continuity).
\begin{definition}
For a semigroup $\T^t$ we define the operator $\A:\D(\A)\to X$ by
\[
\A f = \lim_{t\to 0}\frac{\T^tf-f}{t},\qquad f\in\D(\A),
\]
with $\D(\A)\subset X$ being the linear subspace of $X$ where the
above limit exists. The operator $\A$ is called the
\emph{infinitesimal generator} of the semigroup.
\end{definition}
For $\P^t$, the infinitesimal generator  turns out to be (provided the $F_i$ are continuously differentiable)
\begin{equation}\label{eq:generator}
\A_{PF} f = -\dvg (fF),
\end{equation}
see \cite{LaMa94a}. The following result (see eg.\ Theorem 2.2.4
\cite{Pa83a}) shows the connection between the eigenvalues of the
semigroup operators and their infinitesimal generator:
\begin{theorem}[Spectral mapping theorem]
 Let $\T^t$ be a $C_0$ semigroup and let $\A$ be its infinitesimal generator. Then
 \[ e^{t\sigma(\A)} \subset \sigma\left(\T^t\right) \subset e^{t\sigma(\A)}\cup \{0\}, \]
 where $\sigma(\cdot)$ denotes the point spectrum of the operator. The corresponding eigenvectors are identical.
\end{theorem}
This has important consequences for invariant densities:
\begin{corollary}
 The function $f$ is a invariant density of $\P^t$ for all $t\ge 0$ if and only if $\A_{PF}f=0$.
\end{corollary}
According to the discussion at the end of Section~\ref{subsec.transfer}, we have
\begin{corollary}
 The eigenvalues of $\A_{PF}$ lie on the imaginary axis.
\end{corollary}

\section{The infinitesimal generator, almost-invariant sets, and escape rates}

In this section we discuss almost-invariant sets and escape rates via a spectral analysis of the infinitesimal generator.
As remarked near the end of Section \ref{subsec.transfer}, the $L^1$ spectrum of $\mathcal{P}^t$ lies on the unit circle, and lacks a spectral gap.
Applying the spectral mapping theorem, we see that the spectrum of $\mathcal{A}$ must be pure imaginary.
In the following discussion, to prove formal results, we will add a small random perturbation to $\mathcal{P}^t$.
Later, in the numerics section, we will see that our numerical methods introduce a numerical diffusion that plays the role of creating a spectral gap.

\subsection{Stochastic perturbations} \label{ssec:stoch_pert}

%\todo{OJ: we need to say why we do need to consider a stochastic perturbation.}

In many real world situations, a deterministic model of some physical system is not appropriate. Rather, one should account for the fact that certain external perturbations are present which might be unknown or for which a detailed model would be overly complicated.  Often, it is appropriate to account for these influences by incorporating a \emph{small random perturbation} into the description.  From a theoretical point of view, this even facilitates the analysis of the system: Under certain assumptions, the transfer operator becomes compact.

We therefore now leave the deterministic setting and assume that our dynamical
system described by the ordinary differential equation $dx/dt =
F(x)$ is slightly stochastically perturbed by \textit{white
noise}. An
exact mathematical treatment of this topic would require tools which
are beyond the scope of this work (for an exact derivation see
\cite{LaMa94a}, chapter 11). We are primarily interested in the time
evolution of probability densities. The following material
highlights the relevant formal statements and attempts to point out
the intuition behind them.  Instead of an ordinary differential
equation, we now deal with a stochastic differential equation (SDE)
\begin{equation}\label{eq:SDE}
\frac{dX}{dt} = F(X) + \ep \frac{dW}{dt},
\end{equation}
with $\ep>0$ and $W$ being a $d$-dimensional Wiener process. The
solutions are time-dependent random variables $X(t)$ with values in
$\mathbb{R}^d$. Just as in the deterministic case, we look at the
evolution of density functions;  now the density functions represent
the distribution of random outcomes of the variables $X(t)$. The
density functions $f$ satisfy
\[ \text{Prob}(X(t)\in A) = \int_A f(x,t)\intd x. \]
Assuming the existence of such a density $f(x,t)$, where
$f(\cdot,0)=f_0$ is given, we may again define the transfer operator
as
$\P_\ep^tf_0 := f(t,\cdot)$.
If the vector field $F$ is smooth enough we have following characterization of the density function, the \emph{Fokker-Planck} or \emph{Kolmogorov forward equation}, cf.\ \cite{LaMa94a}, 11.6:
\begin{equation}
 \frac{\partial f}{\partial t} = \frac{\ep^2}{2}\Delta f - \dvg (fF) =: \A_\ep f.
 \label{eq:FPE}
\end{equation}

\begin{proposition}[\cite{Ama83, Lun95}]	\label{prop:AepsGenC0semigroups}
 The operator $\A_\ep$ (with Neumann boundary conditions\footnote{The \textit{dynamical} reason why Neumann boundary conditions are chosen here is discussed later.}) is the infinitesimal generator  of $C_0$ semigroups $\P_{\ep,1}^t$ on $L^1$, $\P_{\ep,2}^t$ on $L^2$, and $\P_{\ep,0}^t$ on $C^0$.
\end{proposition}
We note that the operator $\P_{\ep,1}^t$ is identical to the transfer
operator $\P_{\ep}^t$.

%\todo{GF: I think we need to discuss here what a
%physical measure now is, uniqueness, etc...}

\paragraph{Invariant densities}  Again, we are particularly interested in distributions (densities) which do not change under the evolution.  Those are again given by fixed points of $\P_{\ep,1}^t$ or alternatively, by functions in the null space of $\A_\ep$.  One can show \cite{E.88c} that $\P_{\ep,1}^t$ is compact and thus the null space of $\A_\ep$ is finite dimensional. Furthermore, due to the white noise, the support of the stochastic transition function associated to $\P_{\ep,1}^t$ is unbounded, and the null space of $\A_\ep$ is one-dimensional, i.e.\ there is a \emph{unique} invariant density (see again \cite{E.88c}, Theorem 1), characterizing the long term dynamical behavior of (\ref{eq:SDE}).

\subsection{Almost-invariant sets}
\label{subsec:almost}

We call a set $A\subset M$ \emph{almost-invariant} with respect to a (not necessarily invariant) probability measure $\nu$ (cf. \cite{FrDe03,Fr05}), if
\begin{equation}
 \rho^t_{\nu}(A) := \frac{\nu\left(\Phi^{-t}(A)\cap A\right)}{\nu(A)} \approx 1
     \label{eq:AI}
\end{equation}
for modest times $t$.  The analogous expression for $\Phi$ perturbed by a small random perturbation is
\begin{equation}
 \rho^t_{\nu,\ep}(A) := \frac{\text{Prob}_\nu(X(0)\in A,X(t)\in A)}{\text{Prob}_\nu(X(0)\in A)} \approx 1
     \label{eq:AIeps}
\end{equation}
for modest times $t$.

We can alternatively characterize this property of a set using an infinitesimal representation of $\mathcal{P}_\ep^t$.
To this end let $f\in L^1$ and consider a measurable set $A\subset M$.  We define the functional $\mathcal{A}_{\ep,A}:L^1\supset\D_A(\A_\ep)\to\R$ by
\begin{equation}\label{eq:flow_rate_gen}
\A_{\ep,A}f:= \lim_{t\to 0} \int_A \frac{\mathcal{P}_\ep^tf - f}{t}\; dm,\qquad f\in\D_A(\A_\ep),
\end{equation}
where $\D_A(\A_\ep)$ is the linear subspace of $L^1$ where the above limit exists. Let $f_A := f\chi_A/(\int f\chi_A dm)$.

\begin{proposition}\label{prop:rate}
Let $0\le f\in \D_A(\A_{\ep})$ be the density of the probability measure $\nu$. Then
\begin{equation}
\rho^t_{\nu,\ep}(A) = 1 + (\A_{\ep,A}f_A)t + o(t)
\end{equation}
for measurable $A$ as $t\to 0$.\footnote{For two functions $f,g:\R\to\R$ we say ``${f(t)=o(g(t))}$ as ${t\to0}$'', if $\lim_{t\to0}\tfrac{f(t)}{g(t)}=0$.}
\end{proposition}
\begin{proof}
We have
\begin{align*}
\frac{\rho^t_{\nu,\ep}(A)-1}{t} & = \frac{\text{Prob}_\nu(X(0)\in A,X(t)\in A)-\nu(A)}{t\, \text{Prob}_\nu(X(0)\in A)} \\
& = \frac{\int_{A} \P^t_{\ep}(f\chi_A)\, dm - \int_A f\, dm}{t\,\nu(A)}  = \int_{A} \frac{\P^t_{\ep}f_A - f_A}{t}\, dm,
\end{align*}
where the second equation follows from the fact that $X(0)$ is distributed according to $\nu$ (with density $f$), and $X(0)\in A$; which gives the density $f\chi_A$ for $X(0)$, see Section~\ref{ssec:stoch_pert}. Thus for $t\to 0$ the claim follows.
\end{proof}

Correspondingly, a set $A$ will be almost invariant with respect to the measure with density $f\in L^1$, if
$\A_{\ep,A}f_A\approx 0$.

There is a strong connection between almost-invariant sets and the spectrum of the generator.
The following theorem illustrates this with a simple heuristic for producing almost-invariant sets.
\begin{theorem}\label{thm:ai}
Suppose that the generator $\A_\ep$ possesses a real eigenvalue $\lambda < 0$ with corresponding eigenfunction $f$. Then $\int f \ dm= 0$. Define $A^+ = \{f\ge0\}$, $A^- = \{f<0\}$. Let $f_{A^+},f_{A^-}\in\D_A(\A)$.
Then
\[
\A_{\ep,A^+}|f_{A^+}|+\A_{\ep,A^-}|f_{A^-}|= \lambda.
\]
\end{theorem}
\begin{proof}
By Proposition~5.7 in \cite{DeJu99a} and the spectral mapping theorem we have that
\[
\rho^t_{\nu,\ep}(A^+)+\rho^t_{\nu,\ep}(A^-) = \exp(t\lambda) + 1,
\]
where $\nu$ is the probability measure with density $|f|$.  Using Proposition~\ref{prop:rate} we obtain
$1+\A_{\ep,A^+}|f_{A^+}|t+\A_{\ep,A^-}|f_{A^-}|t = \exp(t\lambda) + o(t)$, i.e.
\[
\A_{\ep,A^+}|f_{A^+}|+\A_{\ep,A^-}|f_{A^-}| = \frac{\exp(t\lambda)-1}{t}+o(1)
\]
and for $t\to 0$ we obtain the claim.
\end{proof}

For $\lambda \approx 0$ Theorem~\ref{thm:ai} yields $\A_{\ep,A^\pm}|f_{A^\pm}|\approx 0$, which means that the sets $A^+$ and $A^-$ will be almost-invariant with respect to the probability measure with density $|f|$.
Other techniques for extracting sets $A^+, A^-$ from the eigenfunction $f$ may be found in \cite{FrDe03,Fr05}.
The papers \cite{FrDe03,Fr05} also discuss almost-invariance with respect to physical \emph{invariant} probability measures, a property that is particularly meaningful when studying typical dynamical behavior.
In all cases, the basis for these methods are eigenfunctions of $\A_\ep$ corresponding to (real) eigenvalues close to $0$.

If $\A_\ep$ has a complex eigenvalue with real part close to zero, then the corresponding complex eigenfunction may also be used to construct almost-invariant sets.
%The following lemma answers the question of how to extract the almost invariant sets from the eigenfunctions of $\A$.
\begin{lemma}		\label{lem:realpart}
  Let $\A_\ep f = \lambda_A f$ with $\lambda_A\in \mathbb{C}$ and let $f_{re}$ and $f_{im}$ denote the real and imaginary part of $f$, respectively. Let $t>0$ be such that $e^{t\lambda_A}=\lambda_{PF}\in\R$. Then $\P_\ep^tf_{re} = \lambda_{PF} f_{re}$ and $\P_\ep^tf_{im} = \lambda_{PF} f_{im}$.
\end{lemma}
\begin{proof}
  From the proof of Theorem~2.2.4 \cite{Pa83a} we have $\P_\ep^tf = \lambda_{PF} f$. Note, that $\P_\ep^t:L^1(M,\R)\circlearrowleft$. By linearity we have $\P_\ep^tf = \P_\ep^tf_{re} + i\ \P_\ep^tf_{im}$. Thus,
  \[
  \underbrace{\lambda_{PF} f_{re}}_{\in\R} + i\ \underbrace{\lambda_{PF} f_{im}}_{\in\R} = \underbrace{\P_\ep^tf_{re}}_{\in\R} + i\ \underbrace{\P_\ep^tf_{im}}_{\in\R}.
  \]
  The claim follows immediately.
\end{proof}

Hence, if for a $t>0$ we have $1\approx e^{t\lambda_A}\in\R$, then the real (and imaginary) part of the corresponding eigenfunction yields a decomposition of the phase space into almost invariant sets in the sense of Theorem~\ref{thm:ai}.

%\todo{
%GF: in the above lemma do you imagine that we may look at
%$\lambda_A$ with $Im\lambda_A\neq 0$ and choose $t$ so that
%$e^{t\lambda_A}\in\mathbb{R}$?  I did not consider such an idea
%before;  this seems to then heuristically determine a particular
%flow duration for which the features in the corresponding
%eigenfunction for $P^t$ recur.  This is in contrast to pure real
%eigenvalues $\lambda_A$ for which the features are almost-invariant
%for all $t\ge 0$.\\PK: Right. For nonzero imaginary parts there is a cyclic behavior involved. For example for the sine flow you dont get any pure real (nonzero) eigenvalues, so there are no almost invariant sets with big AI ratio, but because of periodicity there are some complex EVs with real part near zero. If we do not want to include this kind of behavior, the lemma is unnecessary.}

\subsection{Escape rates}

We have seen in the previous section that there is a connection between eigenvalues of $\mathcal{A}_\ep$ close to zero and almost-invariant sets.
There is also a strong connection between the \emph{escape rate} of the sets $A^+$, $A^-$ and the corresponding eigenvalue of $\mathcal{A}_\ep$.

In the deterministic setting ($\ep=0$) the upper (resp.\ lower) escape rate of a set $A$ under the time-$t$ map of the flow $\Phi^t$ is defined as:
\begin{eqnarray}
\label{eq:upperescape}
\bar{E}(A)&=&-\liminf_{k\to\infty} \frac{1}{k}\log m(A\cap \Phi^{-t}A\cap\cdots\cap \Phi^{-kt}A) \\
\label{eq:lowerescape}
\underbar{E}(A)&=&-\limsup_{k\to\infty} \frac{1}{k}\log m(A\cap \Phi^{-t}A\cap\cdots\cap \Phi^{-kt}A)
\end{eqnarray}
If both limits are equal, we call $E(A)=\lim_{k\to\infty} \frac{1}{k}\log m(A\cap \Phi^{-t}A\cap\cdots\cap \Phi^{-kt}A)$ the escape rate from $A$.
The escape rate is the asymptotic exponential rate of loss of $m$-mass from the set $A$.
One might expect that almost-invariant sets have low escape rate and vice-versa, but simple counterexamples in \cite{FrSt10} show that this is not the case.
This is because the notion of almost-invariance as defined above is a finite-time property, while escape rate is an asymptotic quantity.

The $\ep>0$ version of $(\ref{eq:upperescape})$ is
\begin{eqnarray}
\label{eq:upperescapeep}
\bar{E}_\ep(A)&:=&-\liminf_{k\to\infty} \frac{1}{k}\log \text{Prob}_m(X(0)\in A, X(t)\in A,\ldots,X(kt)\in A)\\
&=&-\liminf_{k\to\infty} \frac{1}{k}\log \int_A (\P^t_{\ep,A})^{k}\left(1\right)\, dm. 
\end{eqnarray}
where $\P_{\ep,A}^t(f):=\P_\ep^t(f\chi_A)$ is the standard restriction of the operator $\P_\ep^t$ to $A$~\cite{pianigiani1979expanding}.
The following theorem provides a recipe for constructing sets with low escape rate from eigenfunctions of $\A_\ep$ with real eigenvalues.

\begin{theorem}
\label{thm:er}
Suppose that $\A_\ep f=\lambda f$ for some $\lambda<0$ and $f\in L^\infty(m)$. Then
\[
\bar{E}_\ep(A^+)\le -t\lambda \mbox{ and }\bar{E}_\ep(A^-)\le -t\lambda,
\]
where $A^+=\{f\ge 0\}$ and $A^-=\{f<0\}$.
\end{theorem}
\begin{proof}
Let $f$ be scaled such that $\int|f|\,dm=2$, and define the (signed) measure $\nu$ as $\nu(A)=\int_A f\,dm$. Then $\nu^+ = \nu\big\vert_{A^+}$ and $\nu^-=-\nu\big\vert_{A^-}$ are both probability measures. For an event $\mathcal{E}$ define $\text{Prob}_{\nu}(\mathcal{E}) := \text{Prob}_{\nu^+}(\mathcal{E}) - \text{Prob}_{\nu^-}(\mathcal{E})$. Since $f$ is an eigenfunction of $\P_{\ep}^t$ with eigenvalue $e^{\lambda t}$, we have
\begin{eqnarray*}
e^{\lambda kt} = e^{\lambda kt}\int_{A^+} f\, dm & = & \int_{A^+}\P_{\ep}^{kt}f\, dm 
							  =  \text{Prob}_{\nu}\big(X(kt)\in A^+\big)\\
							 & = & \text{Prob}_{\nu}\big(X(\ell t)\in A^+, \ell=0,\ldots,k\big) \\
							 &   & + \sum_{n=0}^{k-1} \underbrace{\text{Prob}_{\nu}\big(X(nt)\in A^-, X(\ell t)\in A^+,\ell=n+1,\ldots,k\big)}_{=:p_n}.
\end{eqnarray*}
It follows for the summands $p_n$ in the above sum:
\begin{eqnarray*}
p_n & = & \text{Prob}_{\P_{\ep}^{nt}\nu}\big(X(0)\in A^-, X(\ell t)\in A^+,\ell=1,\ldots,k-n\big) \\
 & = & e^{\lambda nt}\text{Prob}_{\nu}\big(X(0)\in A^-, X(\ell t)\in A^+,\ell=1,\ldots,k-n\big) \\
 & = & -e^{\lambda nt}\text{Prob}_{\nu^-}\big(X(\ell t)\in A^+,\ell=1,\ldots,k-n\big),
\end{eqnarray*}
where in the first equation the action of the transfer operator $\P_{\ep}^t$ on the measure $\nu$ is defined through its action on the density $f$ of $\nu$. The second equation follows from $\nu$ being an eigenmeasure. Clearly, $p_n$ is non-positive, and thus
\[
\text{Prob}_{\nu}\big(X(\ell t)\in A^+, \ell=0,\ldots,k\big) \ge e^{\lambda kt}
\]
holds. The rest of the proof follows the lines of the one of Theorem~2.4 in~\cite{FrSt10}.
\end{proof}

The recipe of Theorem~\ref{thm:er} is in fact the same as in Theorem~\ref{thm:ai}.
The difference is the measure used:  in Theorem~\ref{thm:ai}, almost-invariance is computed with respect to the particular measure $\nu$ with density $|f|$ where $f$ is the eigenfunction in question, whereas in Theorem~\ref{thm:er} escape is computed with respect to Lebesgue measure.

%In Theorem~\ref{thm:er}, we are computing escape rate with respect to Lebesgue measure;  in fact, we are free to compute with respect to any finite measure $\mu$, so long as $m$ is absolutely continuous to $\mu$ and $f\in L^\infty(\mu)$.

\begin{remark}
A more natural notion of escape rates in the time-continuous case would be the following:
\[
\bar{E}_{\ep}(A) := -\liminf_{t\to\infty}\frac{1}{t}\text{Prob}_m\big(X(s)\in A,s\in [0,t)\big).
\]
However, for this definition it is more complicated to obtain similar results to the one in Theorem~\ref{thm:er}, and a fuller discussion will appear elsewhere. One can view~\eqref{eq:upperescapeep} with $t=1$ as the ``time sampled version'' of this continuous-time definition.
\end{remark}

\section{Numerical approximation}
\label{sec:methods}

Having seen that certain eigenpairs of the transfer operator (resp.\ infinitesimal generator) carry the information we are seeking, we describe in the sequel our proposed approximation of these eigenpairs. To this end, we define finite dimensional approximation spaces and consider the eigenvalue problem projected onto these spaces.  Throughout this section we assume that the underlying deterministic vector field $F$ is smooth.

\subsection{Ulam's method}
\label{sec:ulam}
We describe here the ``standard'' Ulam approach;  see the surveys \cite{Fr01,DeFrJu01a} for more details.
We partition $M$ into $d$-dimensional connected, positive volume subsets $\{B_1,\ldots,B_n\}$. Typically, each $B_i$ will be a hyperrectangle or simplex to simplify computations.  As an approximation space we consider the space
$\Delta_n=\mbox{sp}\{\chi_{B_1},\ldots,\chi_{B_n}\}$ of functions which are piecewise constant on the cells of the partition.  Let $\pi_n:L^1\to\Delta_n$,
$
\pi_n f = \sum_{i=1}^n \frac{1}{m(B_i)}\int_{B_i} f\; dm\; \chi_{B_i},
$
be the $L^2$-orthogonal projection onto $\Delta_n$.  We let $\P_n^t:\Delta_n\to\Delta_n$, $\P_{n}^t := \pi_n \mathcal{P}^t$,
be the approximate Frobenius-Perron operator.  Note that
$
\P_{n}^t \chi_{B_i}  = \pi_n\P^t\chi_{B_i}  = \sum_{j=1}^n \frac{1}{m(B_j)}\int_{B_j} \P^t \chi_{B_i}\; dm\; \chi_{B_j},
$
i.e.\ the matrix representation $P_{n}^t\in\R^{n\times n}$ of $\P_{n}^t$ with respect to the basis $\chi_{B_1},\ldots,\chi_{B_n}$ and multiplication on the left is
\begin{align*}
(P_{n}^t)_{ij} & = \frac{1}{m(B_j)}\int_{B_j} \P^t \chi_{B_i}\; dm =\frac{m(B_i\cap\Phi^{-t}B_j)}{m(B_j)}.
\end{align*}
This matrix is easily constructed numerically using eg.\ GAIO \cite{DeFrJu01a}.

In the stochastic setting, letting $\P_{\ep,n}^t:\Delta_n\to\Delta_n$, $\P_{\ep,n}^t := \pi_n \P^t_{\ep}$, one obtains
\begin{align*}
(P_{\ep,n}^t)_{ij} & = \frac{1}{m(B_j)}\int_{B_j} \P_\ep^t \chi_{B_i}\; dm
\end{align*}
as above.
In principle one can use Monte-Carlo integration to compute these entries, however, this is not particularly efficient.

%
%A simple, yet not very efficient (at least in low dimensional phase spaces) way to approximate the entries of this matrix is via Monte-Carlo integration: One chooses a certain number of test points i.i.d. according to a uniform distribution in $B_i$, evolves them according to $\P_\ep$ and counts how many get mapped to $B_j$.
%
%In the deterministic setting ($\ep=0$) and with respect to the basis $\chi_{B_1},\ldots,\chi_{B_n}$, its matrix representation is given by the matrix $P_n^t$ with entries
%\[
%(P^t_{n})_{ij}:=\frac{m(\Phi^{-t}B_j\cap B_i)}{m(B_j)}.
%\]

\subsection{Ulam's method for the generator}

%[Thanh, what follows
%is what one might do without the diffusion term.]
We partition $M$ as in the standard Ulam's method.  We will see that the numerical scheme itself introduces some diffusion and we therefore consider the deterministic generator $\A$ (i.e.\ $\ep=0$).

%[Thanh, this is the standard matrix approximation used to find
%fixed points of $\mathcal{P}$ for time-$t$ flows and maps.]
We wish to construct an operator
$\A_n:\Delta_n\to\Delta_n$ that is close in some sense to the operator $\A$. Motivated by Ulam's method, one would like to form $\pi_n\A$, which unfortunately does not exist, because $\Delta_n\nsubseteq\D(\A)$, cf.\ (\ref{eq:generator}).  Instead of differentiating w.r.t. time and then doing the projection, we swap the order of these operations. Let us build the Ulam approximation $\P_n^t$ first, which will \textit{not} be a semigroup any more, but for fixed $t$ it approximates $\P^t$. Taking the time derivative, our candidate approximate operator is
\begin{equation}
\A_nf:=\lim_{t\to 0}\left(\frac{\pi_n\P^t\pi_nf-\pi_nf}{t}\right).
\label{eq:AnUlam}
\end{equation}
The following lemma emphasizes the intuition behind this definition:
if $\P_n^t$ is a sufficiently good approximation (for small $t$) of
the Markov jump process generated by the dynamics on the sets $B_i$,
then $\A_n$ will be the generator of this process.  To be exact, the following is the case:
\begin{lemma}
\label{ulamformula1lemma}
The matrix representation of
${\mathcal{A}}_n:\Delta_n\circlearrowleft$ with respect to the basis $\chi_1,\ldots,\chi_n$ under multiplication on the left is
\begin{equation}\label{ulamformula}
({A}_n)_{ij}=\left\{
            \begin{array}{ll}
              \displaystyle\lim_{t\to 0}\frac{m(B_i\cap \Phi^{-t}B_j)}{t\cdot m(B_j)}, & \hbox{$i\neq j$;} \\
              \displaystyle\lim_{t\to 0}\frac{m(B_i\cap \Phi^{-t}B_i)-m(B_i)}{t\cdot m(B_i)}, & \hbox{otherwise.}
            \end{array}
          \right.
\end{equation}
\end{lemma}
\begin{proof}
We consider the action of $\P^t$ on
$\chi_{B_i}$.
\begin{eqnarray*}
% \pi_n\mathcal{\mathcal{A}}\chi_{B_i} &=& \pi_n\lim_{t\to 0}\frac{\P^t\chi_{B_i}-\chi_{B_i}}{t} \\
\lim_{t\to 0}\pi_n\frac{\mathcal{P}^t\chi_{B_i}-\chi_{B_i}}{t}&=& \lim_{t\to 0}\sum_{j=1}^n \frac{1}{m(B_j)}\left(\int_{B_j}\frac{\mathcal{P}^t\chi_{B_i}-\chi_{B_i}}{t}\ dm\right)\chi_{B_j}\\
%  \mbox{by Lemma \ref{domconvlemma}}\\
&=& \lim_{t\to 0}\sum_{j\neq i} \frac{1}{m(B_j)}\left(\int_{B_j}\frac{\P^t\chi_{B_i}}{t}\ dm\right)\chi_{B_j} \\
& & + \lim_{t\to 0}\frac{1}{m(B_i)}\left(\int_{B_i}\frac{\P^t\chi_{B_i}-\chi_{B_i}}{t}\ dm\right)\chi_{B_i}\\
&=& \lim_{t\to 0}\sum_{j\neq i} \frac{1}{m(B_j)}\left(\int_{\Phi^{-t}B_j}\frac{\chi_{B_i}}{t}\ dm\right)\chi_{B_j}\\&&\qquad + \lim_{t\to 0}\frac{1}{m(B_i)}\left(\int_{\Phi^{-t}B_i}\frac{\chi_{B_i}}{t}\ dm-\int_{B_i}\frac{\chi_{B_i}}{t}\ dm\right)\chi_{B_i}\\
&=& \sum_{j\neq i} \lim_{t\to 0}\frac{m({B_i\cap \Phi^{-t}(B_j)})}{t\cdot m(B_j)}\chi_{B_j} + \lim_{t\to 0}\frac{{m({B_i\cap \Phi^{-t}B_i})}-m(B_i)}{t\cdot m(B_i)}\chi_{B_i}\\
\end{eqnarray*}
Thus under left multiplication we obtain (\ref{ulamformula}).
\end{proof}
\begin{remark}
 Lemma \ref{ulamformula1lemma} states that $(A_n)_{ij}$ is the \textit{outflow rate of mass} from $B_i$ into $B_j$.
\end{remark}

Let $\mathcal{R}_n^t:=\exp\left(t\A_n\right) = I-\pi_n+\exp(t\A_n\mid_{V_n})\pi_n$ denote the semigroup generated by $\A_n$. Then $\mathcal{R}_n^t$ and $\P^t$ are near to each other in the following sense, cf.\ \cite{Ko10a}:
\begin{proposition}	\label{prop:RnearQP}
 As $t\to 0$
 \begin{equation}
 \mathcal{R}_n^tf - \pi_n\P^tf = \mathcal{O}(t^2)
 \label{eq:RnearQP}
 \end{equation}
 for all $f\in \Delta_n$.\footnote{For two functions $f,g:\R\to\R$ we say ``${f(t)=\mathcal{O}(g(t))}$ as ${t\to0}$'', if $\limsup_{t\to0}\tfrac{|f(x)|}{|g(x)|}<\infty$.}
\end{proposition}
The following lemma allows us to construct $\mathcal{A}_n$ without the computation of the flow $\Phi^t$.
%[Thanh, the following lemma is the punchline since if we can get
%the earlier lemma to work we have a nice representation of the
%discretised operator that is intuitive and cheap to produce
%numerically.]
\begin{lemma}
\label{ulamformula2lemma} For $i\neq j$, define $\mathbf{n}_{ij}$
to be the the unit normal vector pointing out of $B_i$ into $B_j$
if $B_i\cap B_j$ is a $d-1$-dimensional face, and the zero vector
otherwise. The matrix representation of
$\mathcal{A}_n:\Delta_n\circlearrowleft$ with respect to the basis $\chi_1,\ldots,\chi_n$ under multiplication on the left is
\begin{equation}\label{ulamformula2}
 (A_n)_{ij}=\left\{
            \begin{array}{ll}
              \displaystyle \frac{1}{m(B_j)}\int_{B_i\cap B_j}\max\{F(x)\cdot \mathbf{n}_{ij},0\}\ dm_{d-1}(x), & \hbox{$i\neq j$;} \\
              \displaystyle -\sum_{j\neq i} \frac{m(B_j)}{m(B_i)}(A_n)_{ij}, & \hbox{otherwise.}
            \end{array}
          \right.
\end{equation}
\end{lemma}
\begin{proof}
From (\ref{ulamformula}) we have for $i\neq j$ that
$A_{n,ij}=\lim_{t\to 0}\frac{m(B_i\cap \Phi^{-t}B_j)}{t\cdot
m(B_j)}$. Denoting $M_{ij}(t)=m(B_i\cap \Phi^{-t}B_j)$ we have that
${A}_{n,ij}=M'_{ij}(0)/m(B_j)$ where the prime denotes
differentiation with respect to $t$. The quantity $M'_{ij}(0)$ is
simply the rate of flux out of $B_i$ through the face $B_i\cap
B_j$ into $B_j$ and so $M'_{ij}(0)=\int_{B_i\cap B_j}
\max\{F(x)\cdot \mathbf{n}_{ij},0\} dm_{d-1}(x)$.

For the diagonal elements ${A}_{n}$ we have
${A}_{n,ii}=\lim_{t\to 0}\frac{m(B_i\cap
\Phi^{-t}B_i)-m(B_i)}{t\cdot m(B_i)}$. Note that $m(B_i)-m(B_i\cap
\Phi^{-t}B_i)=m(B_i\setminus \Phi^{-t}B_i)$. Clearly $B_i\setminus
\Phi^{-t}B_i=B_i\cap\bigcup_{j\neq i}\Phi^{-t}B_j=\bigcup_{j\neq i}
B_i\cap \Phi^{-t}B_j$ modulo sets of
Lebesgue measure zero. %In the remainder of the proof we imagine that
%the shared boundary points of the sets $A_1,\ldots, A_n$ are
%reassigned so that the collection $\{A_1,\ldots,A_n\}$ is a true
%partition of $X$.  This reassignment only affects the shared
%boundary points which are a set of Lebesgue measure zero. Firstly,
%if $x\in A_i\setminus\Phi^{-t}A_i$ then $x\in A_i$ and $x\notin
%\Phi^{-t}A_i$. The latter implies $x\in \Phi^{-t}A_j$ for some
%$j\neq i$;  thus $x\in\bigcup_{j\neq i} A_i\cap\Phi^{-t}A_j$.
%Secondly, if $x\in \bigcup_{j\neq i} A_i\cap\Phi^{-t}A_j$ then in
%particular, $x\notin \Phi^{-t}A_i$ and so $x\in
%A_i\setminus\Phi^{-t}A_i$.
Thus, $m(B_i)-m(B_i\cap \Phi^{-t}B_i)=\sum_{j\neq i} m(B_i\cap \Phi^{-t}B_j)$.
Now, by (\ref{ulamformula}), $A_{n,ii}=-\lim_{t\to 0}\frac{\sum_{j\neq i} m(B_i\cap \Phi^{-t}B_j)}{m(B_i)}=-\sum_{j\neq i} \frac{m(B_j)}{m(B_i)}A_{n,ij}.$ %
%
% Denote
%$M_{ii}(t)=m(B_i)-m(A_i\cap\Phi^{-t}A_i)$;  then
%$A_{ii}=-M'_{ii}(0)$. Note that $M'_{ii}(0)$ is the rate of flux
%leaving $A_i$ through all faces of $A_i$. Thus
%$M'_{ii}(0)=\sum_{j:A_j\cap A_i\mbox{ is a $d-1$-dimensional face}}
%\int_{A_i\cap A_j} \max\{F(x)\cdot \mathbf{n}_{ij},0\} dm_{d-1}(x)$.
%We are done.
\end{proof}

In one dimension, (\ref{ulamformula2}) has a particularly simple form.
\begin{corollary}
\label{onedcorollary}
Let $M=\mathbb{T}^1$. Assume (without loss\footnote{If $F\ngeq 0$ and $F\nleq 0$, we have one or more stable fixed points, and every trajectory converges to one of them. Hence, there is no interesting statistical behavior to analyze.}) that $F(x)\ge 0$. Let $0=x_0 < x_1 < \cdots < x_n=1$ and consider the partition $\{B_1,\ldots,B_n\}$ with $B_i=[x_{i-1},x_i]$. Then
\begin{equation}\label{ulam1dformula}
(A_n)_{ij}=\left\{
            \begin{array}{ll}
              F(x_i)/m(B_j), & \hbox{$j=i+1$;} \\
              -F(x_i)/m(B_i), & \hbox{$j=i$;} \\
              0, & \hbox{otherwise.} \\
                  \end{array}
          \right.
\end{equation}
\end{corollary}

\begin{remark}[Connections with the upwind scheme]	\label{rem:upwind}
%What we derived here is the so-called upwind scheme \cite{Le92a} in finite volume methods. The scheme is known to be stable. Stability in finite volume schemes is often related to numerical diffusion. Our derivation allows the understanding of stability for these schemes in a similar way. Consider $\P_n^t$, which is the transition matrix of a Markov process near the Markov jump process generated by $\A_n$. $\P_n^t$ can be related to a dynamical system, which after mapping the initial point adds some uncertainty to produce a uniform distribution of the image point in the box $B$ where it landed; see \cite{Fr96a}. One may consider this uncertainty as the numerical diffusion that yields stability.
Clearly, $\A_n$ is the spatial discretization from the so-called upwind scheme in finite volume methods; cf.~\cite{Lev02}. The scheme is known to be stable. Stability of finite volume schemes is often related to ``numerical diffusion'' in them.

We demonstrate numerical diffusion on a simple example, for further details we refer to~\cite{Lev02} Section~8.6.1. Set $M=\mathbb{T}^1$ and $F(x)=\hat F>0$ $\forall x\in M$. Then $\A f=-\hat F f'$ and $\A_n$ is the backward difference scheme. Consider $f_n:=\pi_nf$ as a vector of values. For sufficiently smooth $f$ we have
\begin{eqnarray*}
(\A_n f)_i=n\hat F\left(f_{n,i-1}-f_{n,i}\right)&=&\left(\pi_n\A f\right)_i + \mathcal{O}\left(n^{-1}\right),
\end{eqnarray*}
but
\begin{eqnarray*}
(\A_n f)_i=n\hat F\left(f_{n,i-1}-f_{n,i}\right)
						 & = & n\hat F\left(\frac{f_{n,i-1}-f_{n,i+1}}{2}+\frac{f_{n,i-1}-2f_{n,i}+f_{n,i+1}}{2}\right) \\
						 & = & \hat F\frac{f_{n,i-1}-f_{n,i+1}}{2n^{-1}}+\frac{\hat F}{2n}\frac{f_{n,i-1}-2f_{n,i}+f_{n,i+1}}{n^{-2}} \\
						 & = & \left(\pi_n\left(\A f+\frac{\hat F}{2n}\Delta f\right)\right)_i + \mathcal{O}\left(n^{-2}\right),
\end{eqnarray*}
hence $\A_n f$ is a better approximation of $\A_{\ep}f$, with $\tfrac{\ep^2}{2}=\tfrac{\hat F}{2n}$, and $\A_{\ep}$ is the infinitesimal generator associated with the SDE $\tfrac{dx}{dt} = F(x)+\ep\tfrac{dW}{dt}$ than of $\A f$. That is why one expects quantities computed by $\A_n$ to reflect the actual behavior of $\A_{\ep}$. Since the above equations contain only \textit{local} estimates, for a general nonconstant $F$ is the diffusion also spatially varying; i.e.\ $\tfrac{\ep(x)^2}{2}=\tfrac{F(x)}{2n}$. Note, that this notion of numerical diffusion is easily extended to multiple space dimensions.

The definition \eqref{eq:AnUlam} allows us to interpret the  numerical diffusion in yet another way. We showed in Proposition~\ref{prop:RnearQP} that $P_n^t$ is the transition matrix of a Markov process near the Markov jump process generated by $A_n$ for small $t>0$. The discretized FPO $P_n^t$ can be related to a non-deterministic dynamical system, which, after mapping the initial point, adds some uncertainty to produce a uniform distribution of the image point in the box where it landed; see~\cite{Fr05b}.
Thus, this uncertainty resulting from the numerical discretization, equivalent to the numerical diffusion in the upwind scheme, can be viewed as the reason for robust behavior --- stability.
\end{remark}

Finally, we show that our constructions (\ref{ulamformula2}) and (\ref{ulam1dformula}) always provide a solution to the system $\mathcal{A}_nf=0$ for some $f\in\Delta_n$.
\begin{lemma}
\label{existenceofzerond} There exists a nonnegative, nonzero
$f\in \Delta_n$ so that $\mathcal{A}_nf=0$.
\end{lemma}
\begin{proof}
Let $D_{n,ij}=m(B_i)\delta_{ij}$ and note that
$Q_n:=D_n^{-1}A_nD_n$ satisfies
\begin{equation}
\label{simmatrix}
(Q_n)_{ij}=\left\{
            \begin{array}{ll}
              \frac{m(B_j)}{m(B_i)}A_{ij}, & \hbox{$i\neq j$;} \\
              -\sum_{j\neq i} \frac{m(B_j)}{m(B_i)}(A_n)_{ij}, & \hbox{otherwise.}
            \end{array}
          \right.
\end{equation}
Note that all row sums of $Q_n$ equal zero.
Let $c=\sum_{j\neq i} Q_{n,ij}$. The matrix
$\hat{Q}_n:=Q_n+cI$ is nonnegative with all row sums equal to $c$.
By the Perron-Frobenius theorem \cite{BePl}, the largest eigenvalue of
$\hat{Q}_n$ is $c$ (of multiplicity\footnote{If $\hat{Q}_n$ is aperiodic (there exists $k$ such
that $\hat{Q}_n^k>0$) then the eigenvalue $c$ is simple and the corresponding eigenvector is positive (see \cite{BePl}).} possibly greater than 1) and one of the corresponding left eigenvectors pair $u_n$ is nonnegative. Clearly $u_n$ is a left
eigenvector of $Q_n$ corresponding to the eigenvalue 0 and
thus $D_nu_n$ is a nonnegative left eigenvector
corresponding to 0 for $A_n$.
\end{proof}
%\begin{example}[Double Gyre]
%\begin{figure}[hbt]
%\includegraphics[width=10cm]{dgyre_t025_eval2.png}\\
%\caption{Double Gyre}\label{doublegyrefig}
%\end{figure}
%\end{example}
%\subsection{Convergence as $n\to\infty$}
%As $n\to\infty$, the estimates $f_n\ge 0$ will at least converge
%weakly to some probability measure. I would like to show that
%$f_n\to\tilde{f}$ so that $A(\tilde{f})=0$.
%To do this, I would
%need to show eg. $\|A_nf-Af\|_{L^1}\to 0$ for
%$f\in\mathcal{D}(A)$.

\subsection{Spectral collocation for the generator}

The eigenfunctions of $\A_\ep$ are smooth. This motivates the use of smooth approximation functions, e.g. polynomials.  We here outline the general principles of spectral collocation methods;  for a more thorough presentation we refer to \cite{CaHuQu91a}, \cite{Boyd01} or \cite{Tr00a}.

Choose a family of approximation spaces $\{V_n\}_{n\in\N}$, such that $V_n\subset C^{\infty}(M)$ for all $n$.  Depending on the type of the phase space, we use two different approximation spaces. We introduce them in one dimension; the multidimensional ones can then be constructed by tensor products.  In both cases, the approximation space comes with an associated set of collocation nodes:
\begin{itemize}
\item \emph{Periodic domain/uniform grid.} We have $M = \mathbb{T}^1$ and restrict ourselves to odd values of $n$. Then the basis we choose for $V_n$ is
\[
\left\{e^{ikx}\right\}_{-n/2-1\leq k \leq n/2}.
\]
The associated collocation nodes are the uniform grid $\{0,1/n,\ldots,(n-1)/n\}$.
\item \emph{Standard domain/Chebyshev grid.} Here, $M = [-1,1]$. The space $V_n$ is spanned by the monomials of order $0$ to $n$. We use Chebyshev polynomials as basis functions:
\[
\left\{\cos\left(k\ \textrm{arccos}\left(x\right)\right)\right\}_{0\leq k\leq n},
\]
together with the Chebyshev grid $\{-\cos(2\pi j/n), j=0,\ldots,n\}$, as collocation nodes.
\end{itemize}
Let $f\in V_n$ and $\mathcal{I}_n:C^\infty\to V_n$ be the interpolation operator for the given collocation nodes.  We define the approximate generator by
\[
\A_{\ep,n}f := \mathcal{I}_n\A_\ep f.
\]
For both cases we have following:
\begin{theorem}[Spectral accuracy, \cite{CaHuQu91a}]\label{thm:spectral}
 For $f\in C^{\infty}(M)$ let $f_n$ be the best approximation of $f$ in $V_n$ w.r.t. the supremum norm $\|\cdot\|_{\infty}$. Then for each $k\in\N$ there is a $c_k>0$ such that
 \begin{equation}
     \|f-f_n\|_{\infty}\leq c_k\ n^{-k}\qquad\text{for all }n\in\N.
 \label{eq:bestapp}
 \end{equation}
\end{theorem}

Convergence then follows from standard results on the analysis of Galerkin methods for elliptic differential operators and spectral approximation (cf.\ also \cite{Ko10a}):
\begin{theorem}
 Let $M$ be a compact tensor product domain with smooth boundary.  Let $F$ be a smooth vector field on $M$ and $\A_\ep$, $\A_{\ep,n}$ defined as above, with Neumann (resp. periodic) boundary conditions. Then the eigenvalues and eigenfunctions of $\A_\ep$ are approximated by the corresponding eigenvalues and eigenfunctions of $\A_{\ep,n}$ with spectral accuracy.
\end{theorem}

\paragraph{Why Neumann boundary conditions?} The objects that we are approximating are densities, so there should be no ``loss of mass'' as time goes on. In a closed physical system, the flow normal to the boundary is zero. Hence, there is no loss of mass through advection. The physical meaning of diffusion is a flow of mass in the direction opposite to the gradient and proportional to its magnitude. Thus, no loss of mass via diffusion translates into Neumann boundary conditions for the densities: their normal gradients have to be zero at the boundary. Equivalently, one may write \eqref{eq:FPE} as a continuity equation,
 \[ \partial_tf = \dvg(I), \]
 with $I:=\tfrac{\ep^2}{2}\nabla f-fF$ being the probability flow or probability current. The condition $I|_{\partial M}=0$ leads precisely  to Neumann boundary conditions.

%============================
%SECTION: COMPUTATIONAL ISSUES
%============================
\subsection{Algorithms}

\quad
\begin{algorithm}[Ulam's method]
\label{ulamalg}
\quad
\begin{enumerate}
\item Partition $M$ into connected sets $\{B_1,\ldots,B_n\}$ of positive volume.  Typically, each $B_i$ will be a hyperrectangle or simplex.
\item Choose $t>0$ and compute the matrix $P^t_n$ or $P^t_{\ep,n}$ as described in Section \ref{sec:ulam}.
   %
%\begin{align*}
%(P_{n}^t)_{ij} & = \frac{1}{m(B_j)}\int_{B_i} \P_t \chi_{B_j}\; dm \\
%& = \frac{1}{m(B_j)}\int_{B_j} \chi_{B_i}(\Phi^{t}(x))\; dm(x) \\
%\end{align*}
%where the integral is estimated by to be computed by an appropriate quadrature formula, for example Monte-Carlo \cite{Hu94a} or a recursive exhaustion technique \cite{GuDeKr97a}.
\item  Estimates of invariant densities for $\Phi^t$ are given by fixed points of $P_{n}^t$ (or $P_{\ep,n}^t$): Let $vP_{n}^t =v$.  Then $f:=\sum_{i=1}^nv_i\chi_{B_i}$ satisfies $\P_{n}^t f = f$.
\item Similarly, eigenvectors of $P_{n}^t$ corresponding to real eigenvalues $\lambda \approx 1$ provide information about almost-invariant sets (cf.\ Theorem~\ref{thm:ai}) and sets with low escape rates (cf.\ Theorem \ref{thm:er}).
\end{enumerate}
\end{algorithm}
Note that $P_{n}^t$ typically is a sparse matrix and the number of nonzero entries per column will be determined by Lipschitz constants of $\Phi^t$.

\begin{algorithm}[Ulam's method for the generator]
\label{ulamalg_gen}
\quad
\begin{enumerate}
\item Partition $M$ into connected sets $\{B_1,\ldots,B_n\}$ of positive volume.  Typically each $B_i$ will be a hyperrectangle or simplex.
\item Compute\begin{equation*}
\renewcommand{\arraystretch}{2.0}
 (A_n)_{ij}=\left\{
            \begin{array}{ll}
            \displaystyle  \frac{1}{m(B_j)}\int_{B_i\cap B_j}\max\{F(x)\cdot \mathbf{n}_{ij},0\}\ dm_{d-1}(x), & \hbox{$i\neq j$,} \\
            \displaystyle -\sum_{j\neq i} \frac{m(B_j)}{m(B_i)}(A_n)_{ij}, & \hbox{otherwise,}
            \end{array}
          \right.
\end{equation*}
where one uses standard quadrature rules to estimate the integral.
\item Estimates of invariant densities lie in the null space of $A_n$.  Let $vA_n =0$ (the existence of such a $v$ is guaranteed by Lemma \ref{existenceofzerond}), then $f:=\sum_{i=1}^nv_i\chi_{B_i}$ satisfies $\mathcal{A}_nf=0$.
\item Similarly, eigenvectors of $A_n$ corresponding to large real eigenvalues $\lambda<0$ provide information about almost-invariant sets and sets with low escape rates.
\end{enumerate}
\end{algorithm}

Note that the discretized generator $A_n$ is a sparse matrix since $A_{n,ij}=0$ if $B_i$ and $B_j$ do not share a boundary.

\begin{algorithm}[Spectral collocation for the generator]
\quad
\begin{enumerate}
	\item Set up a grid $x_0,\ldots,x_n$, which is the Chebyshev grid $x_j = -\cos\left(2\pi j/n\right)$, $j=0,\ldots,n$, if $M=[-1,1]$, or the equispaced grid $x_j=j/n$, if $M=\mathbb{T}^1$. (If the state space $M$ is an affine transformation of the above ones, the grid is transformed analogously.) For multidimensional tensor product spaces the grid is the tensor product of the one dimensional grids.
	\item Let $\{\ell_0,\ldots,\ell_n\}$ denote the Lagrange basis on the grid, i.e.\ all $\ell_i$ are polynomials of maximal degree $n$ with $\ell_i(x_j)=\delta_{ij}$. Denote the interpolation on the grid by $\mathcal{I}_n$, and define the discretization matrix obtained by collocation as
	\[
	(A_{\ep,n})_{ij} = (\mathcal{I}_n\A_\ep\ell_j)(x_i).
	\]
	Note, that the computation of these matrix entries is very simple. If $M=\mathbb{T}^1$, we may switch between the evaluation space and the frequency space simply by the fast Fourier transformation. Multiplication by the vector field $F$ is pointwise multiplication in the evaluation space, computing the derivative  is a diagonal scaling in the frequency space. This can be simply extended to $M=[0,1]$ (cf.\ \cite{Tr00a} Chapter~8), as well as for a multidimensional phase space.
	\item Compute the left eigenvector $v$ of $A_{\ep,n}$ at the eigenvalue of smallest magnitude. Then, $\sum_{i=0}^nv_i\ell_i$ approximates the invariant density.
	\item Similarly, eigenvectors of $A_{\ep,n}$ corresponding to large real eigenvalues $\lambda < 0$ provide information about almost-invariant sets.
\end{enumerate}

\end{algorithm}

\subsection{Solving the eigenproblem}

Once the matrix approximation of the operator or the generator has been computed, we then have to solve the corresponding eigenvalue problem.
In Ulam's method one seeks dominant eigenvalues, and standard Arnoldi type methods easily provide the interesting part of the spectrum (we use the \texttt{eigs} function in Matlab).

For the approximate generator, we look for eigenvalues with small
modulus.  Here, the Arnoldi iteration in \texttt{eigs} uses inverse iteration, i.e.\ repeatedly solves linear systems of the type $A_n u=b$.
Since the matrix $A_n$ is singular, this is not a well posed problem
and we might expect some difficulties with these iterative methods.
A possibility for overcoming them is to solve the eigenvalue problem
for $\bar{A}_n:=A_n-\sigma I_{n\times n}$, with a suitable shift
$\sigma$. This merely introduces a shift of the spectrum and the
eigenfunctions stay the same. For $n$ big enough\footnote{If $A_n$
is obtained from the Ulam-based approach, the spectrum of $A_n$, which
is the generator of a Markov jump process, is in the left complex
half plane. This does not hold in general for the spectral collocation
approach, here we rely on the fast convergence of eigenvalues.}, the
spectrum of $\bar{A}_n$ is going to be well separated from zero.
The size $n$ of the linear system may be large, so a direct solution
via LU factorization may not be feasible. However, in Ulam's method for the generator, $A_n$ is a sparse matrix, hence iterative methods
(e.g.\ GMRES) can be used.

For spectral collocation, in general, the matrix $A_n$ will not be sparse. Nevertheless, in the case where the domain is periodic, the matrix-vector product $A_n x$ can be computed fast using the FFT.

\subsection{Computational complexity}

For $d$ dimensional systems in general, \eqref{ulamformula2} shows us that setting up the approximate generator requires the computation of $(d-1)$- dimensional integrals, while to set up the Ulam-matrix, we compute $d$-dimensional integrals. Of course, the gain is biggest in one dimension (zero dimensional integrals are one-point-evaluations).
Equally importantly, \eqref{ulamformula2} does not require integration of the vector field.

For Ulam's method and Ulam's method for the generator, the matrix entries should be computed to an accuracy $\mathcal{O}(n^{-1/d})$; otherwise we cannot expect the approximate eigenfunctions to exploit the full potential of the approximation space (which is an error of $\mathcal{O}(n^{-1/d})$ for $n$ partition elements in a $d$ dimensional space). Assuming the quadrature rule used to compute the entries to also suffer from the curse of dimension, we expect to need $\mathcal{O}(n)$ and $\mathcal{O}(n^{(d-1)/d})$ function evaluations for each matrix entry in the Ulam-matrix, and the matrix arising from Ulam's method for the generator, respectively. This explains the third column of Table~\ref{tab:error_sources} below.

\section{Examples}

A collection of examples is presented in the following. The examples range
over three phase space dimensions, different boundary types, and chaotic and regular dynamics.
We find that the generator approaches are typically superior to the standard Ulam approach, producing more accurate results with less computation time.  If the vector field is infinitely differentiable, the spectral collocation approach can be extremely accurate.  In our final example, we find that for flows on complicated attractors, the standard Ulam method can take advantage of the attracting dynamics to produce better estimates of invariant densities than the generator approaches.  In all cases, the operator based methods are far superior to direct trajectory integration for estimating the invariant density. Other structures such as almost-invariant sets can only be identified using the operator approaches.

%===========================
%EXAMPLE: A FLOW ON A CIRCLE
%===========================
\subsection{A flow on the circle}		\label{ssec:sineflow}
We start with a one dimensional example, a flow on the unit circle. The vector field is given by
\[
F(x) = \sin(4\pi x)+1.1,
\]
$x\in \mathbb{T}^1 = [0,1]$ with periodic boundary conditions, and we wish to compute the invariant density of the system.  Recall that an invariant density $f\in L^1(\mathbb{T}^1)$ needs to satisfy $\A f = 0$, where $\A f(x) = -\intd(fF)(x)/\intd x$.  The unique solution to this equation is $f^{*}(x)=C/F(x)$, $C$ being a normalizing constant (i.e. such that $\|f\|_1=1$). The honours thesis \cite{Stancevic} numerically investigated the estimation of the invariant density for this flow, and other two-dimensional flows, by finding approximate eigensolutions of the infinitesimal generator using finite difference and finite element methods.
Here we use the three methods discussed in the previous section to approximate $f^*$ and compare their efficiency relative to one another and a histogram of a long simulation:
\begin{enumerate}
\item the classical method of Ulam for the Frobenius-Perron operator,
\item Ulam's method for the generator and
\item spectral collocation for the generator.
\end{enumerate}
As we have a periodic domain and the vector field is infinitely smooth, we expect spectral collocation to perform very well.
Figure \ref{fig:sineflow} shows the true invariant density (dashed line), together with its approximations by the four methods.

\subsubsection{Matlab code}

In this 1D case, the three methods can each be realized in a few lines of Matlab.  For illustration purposes, we include the code here.

\paragraph{Ulam's method}
\begin{Verbatim}[frame=single,fontsize=\footnotesize]
F = @(x) sin(4*pi*x)+1.1;                          % definition of the vector field
n = 32; x = 1/(2*n^2):1/(n^2):1-1/(2*n^2);         % nodes for spatial integration
[t,y] = ode23(@(t,x) F(x),[0 2/n],x');             % time integration of the nodes
I = ceil(max(n*mod(y(end,:),1),1));                % cell indices of image points
J = reshape(ones(n,1)*(1:n),1,n*n);                % construction of transition matrix
P = sparse(I,J,1/n,n,n);
[v,d] = eig(full(P));                              % spectrum of transition matrix
f_n = @(x) abs(v(ceil(max(x*n,1)),1));             % plot approximate inv. density
L1 = quadl(f_n,0,1); fplot(@(x)f_n(x)/L1,[0,1])
\end{Verbatim}

The error in Ulam's method decreases like $\mathcal{O}(n^{-1})$ for smooth invariant densities \cite{DiDuLi93a}.  Thus, we need to compute the transition rates between the intervals to an accuracy of $\mathcal{O}(n^{-1})$ (since otherwise we cannot expect the approximate density to have a smaller error). To this end, we use a uniform grid of $n$ sample points in each interval.  This leads to $\mathcal{O}(n^2)$ evaluations of the vector field.
For the numbers in Figure~\ref{fig:sineflow} we only counted each point once, i.e.\ we neglected the fact that for the time integration we have to perform several time steps per point.
\begin{figure}[h]
 \centering
     \includegraphics[width=0.35\textwidth]{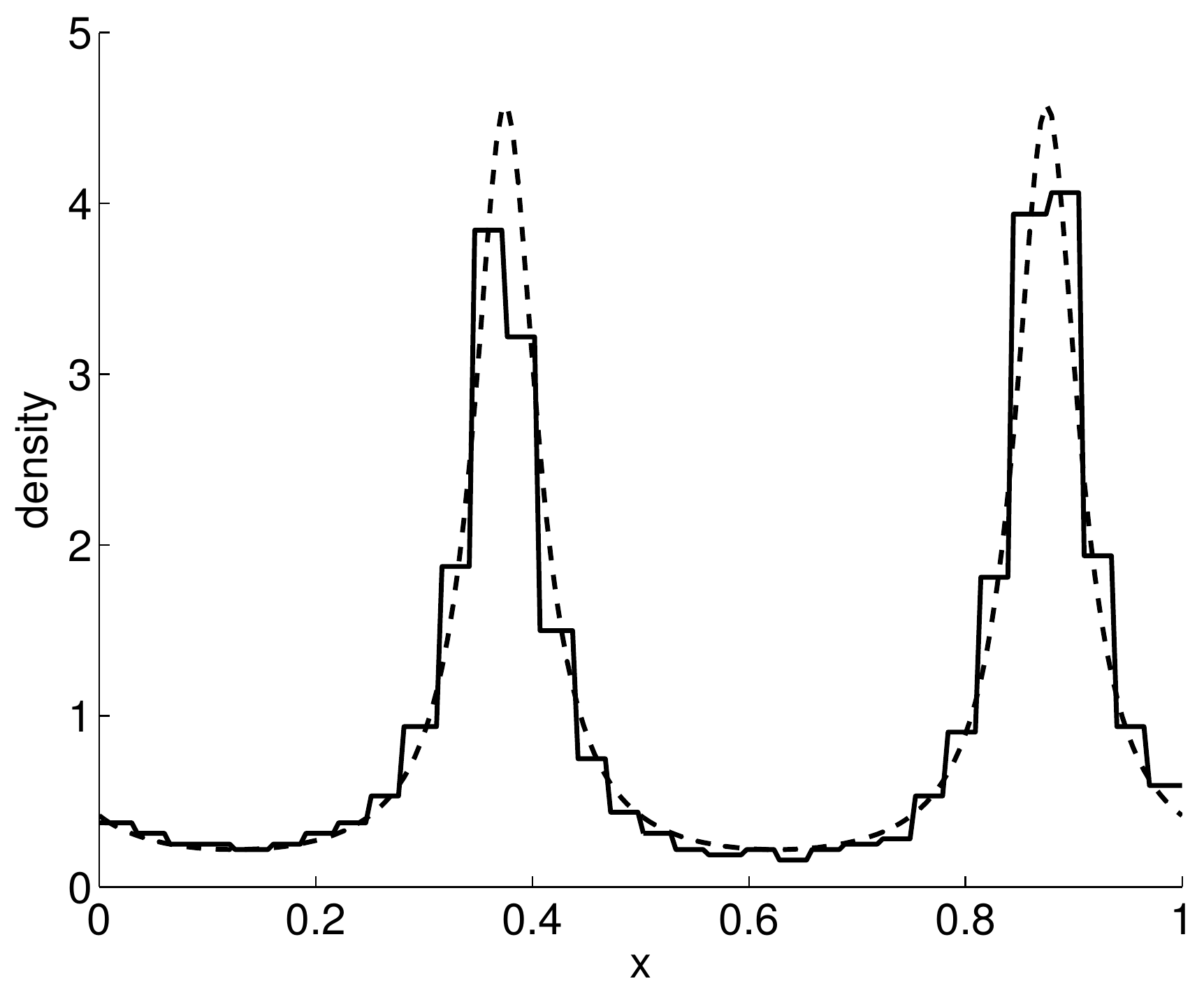}\qquad
     \includegraphics[width=0.35\textwidth]{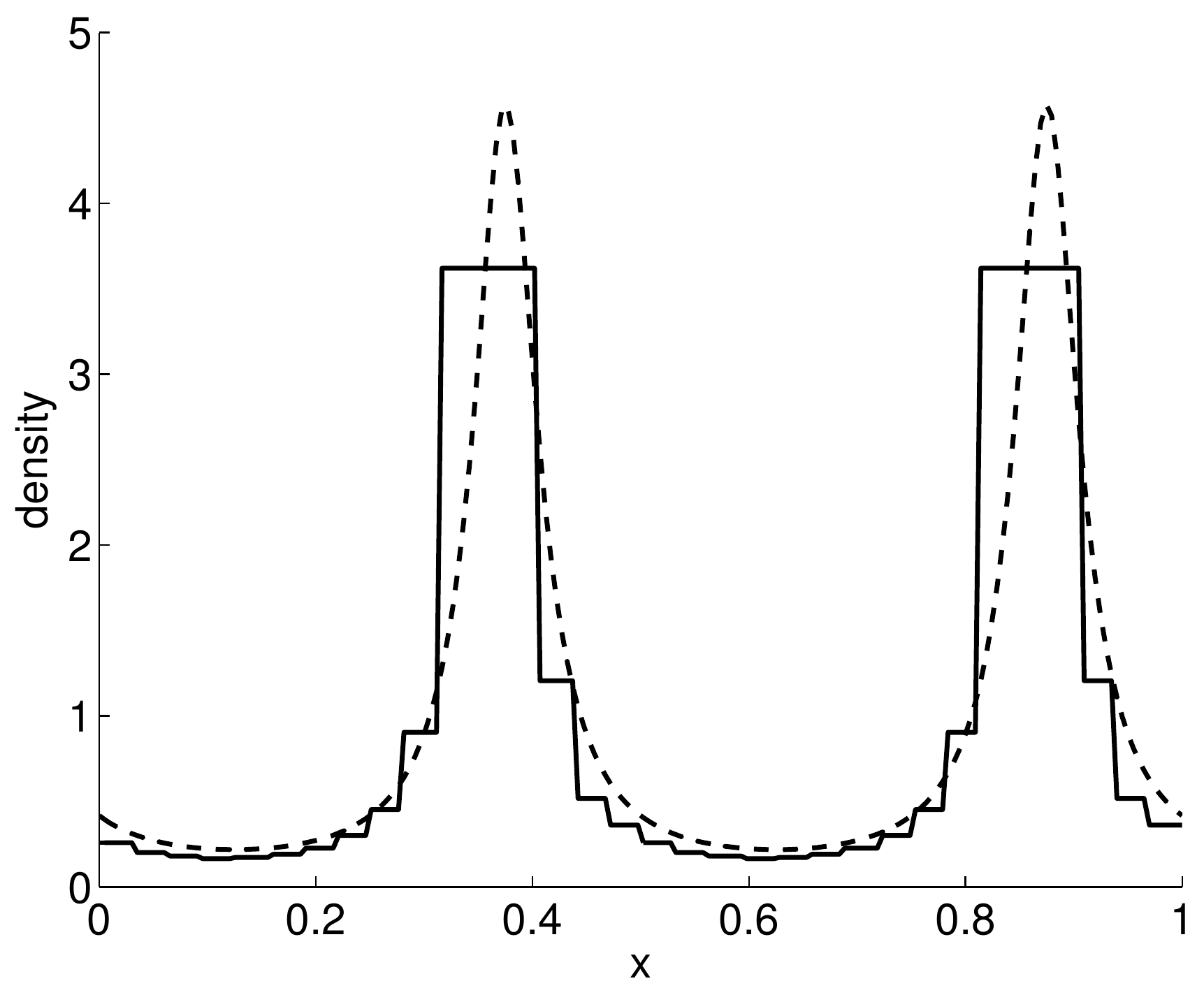}
     \includegraphics[width=0.35\textwidth]{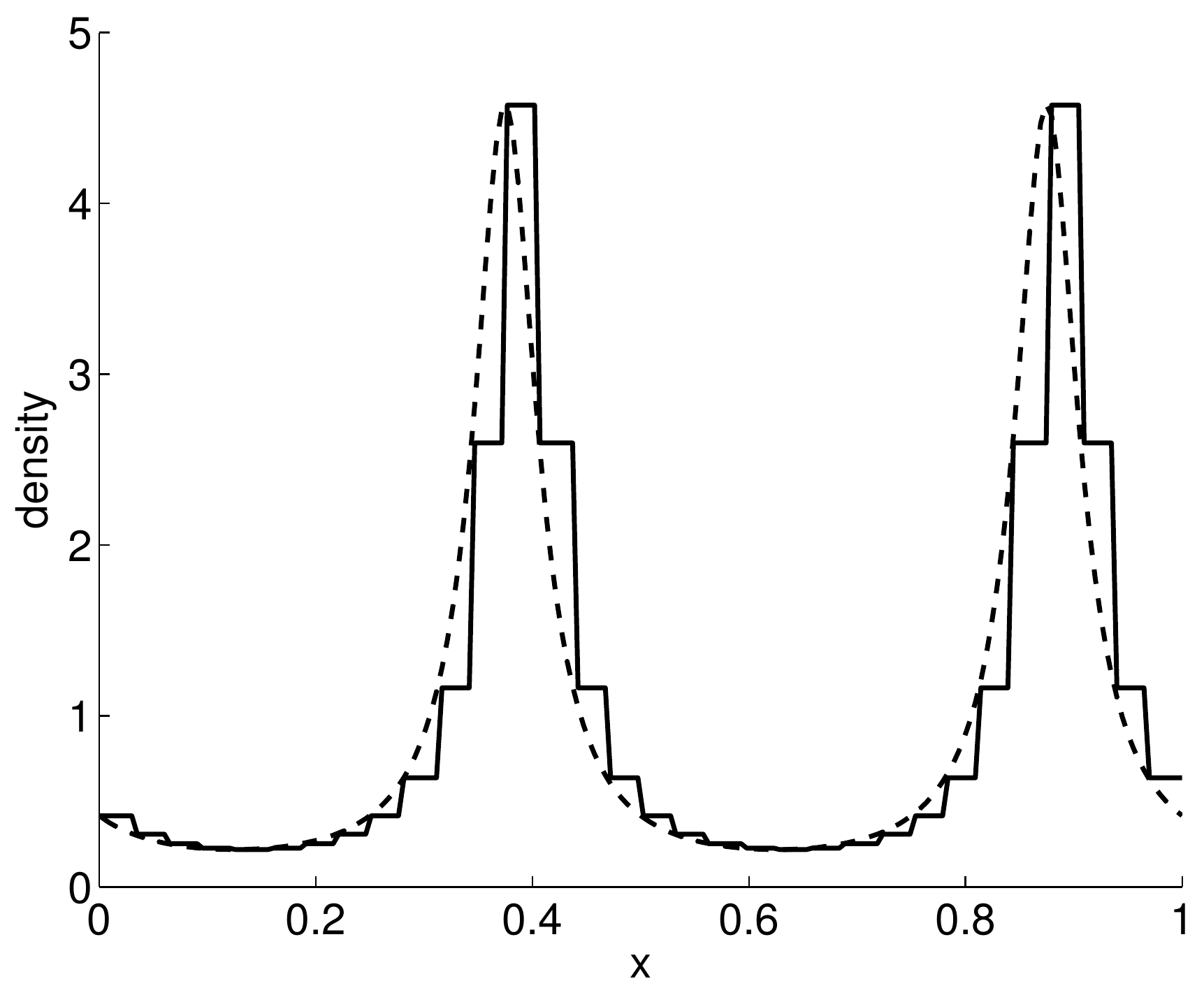}\qquad
     \includegraphics[width=0.35\textwidth]{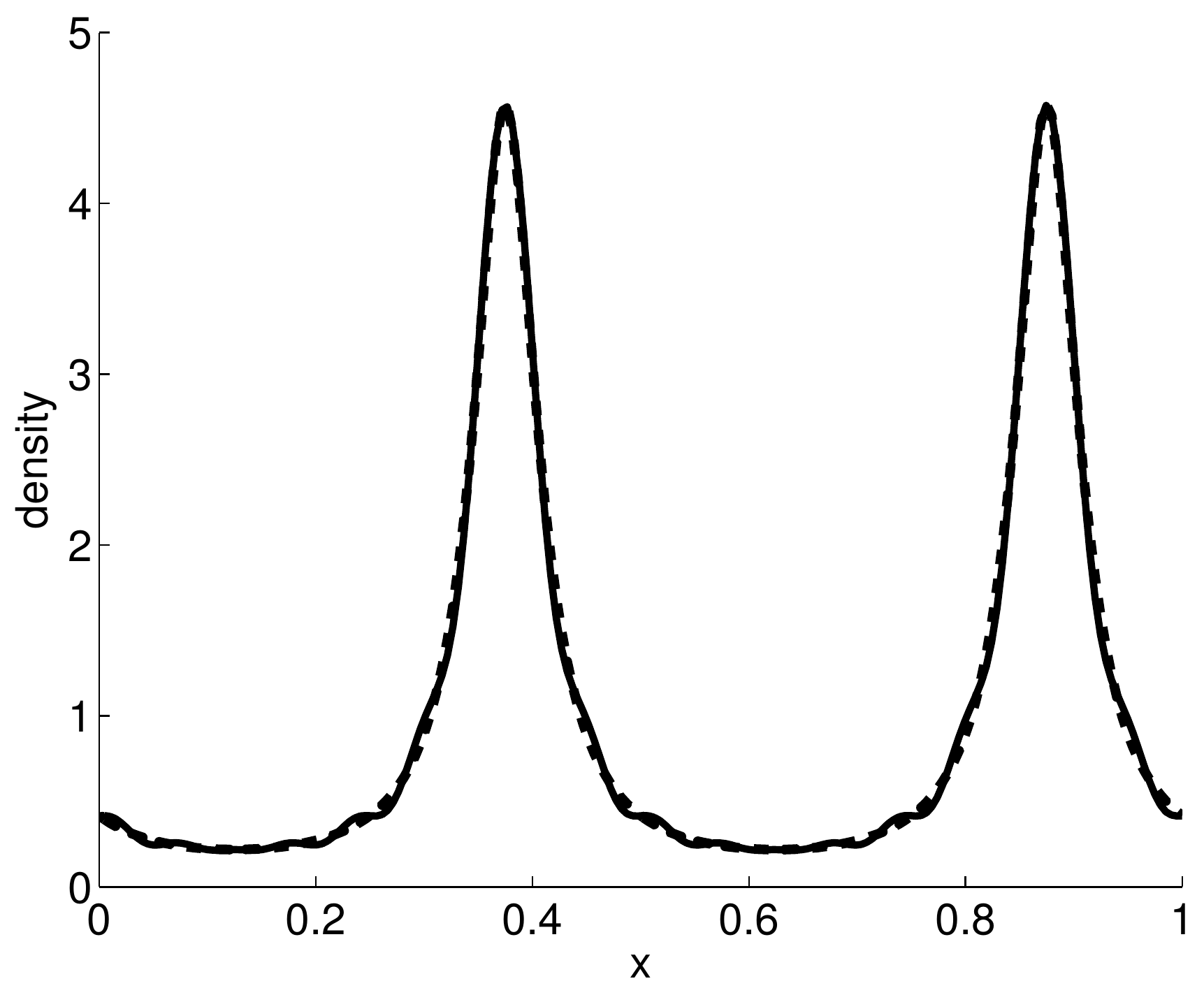}
 \caption{True invariant density (dashed line), approximation by histogramming a long simulation (top left, 1024 iterates), Ulam's method (top right, 32 sample points per partition element, 32 partition elements), Ulam's method for the generator (bottom left, 32 partition elements) and spectral collocation (bottom right, 32 collocation points).}	
 \label{fig:sineflow}
 \end{figure}

\paragraph{Ulam's method for the generator}

\begin{Verbatim}[frame=single,fontsize=\footnotesize]
F = @(x) sin(4*pi*x)+1.1;                           % definition of the vector field
n = 32; Fx = F(0:1/n:1-1/n);                        % evaluation on the boundary nodes
A = n*sparse(1:n,1:n,-Fx) + sparse(1:n,[2:n 1],Fx); % assembling the discrete generator
[v,d] = eig(full(A'));                              % spectrum
[d,I] = sort(diag(d));
f_n = @(x) abs(v(ceil(max(x*n,1)),I(1)));           % plot approximate inv. density
L1 = quadl(f_n,0,1); fplot(@(x)f_n(x)/L1,[0,1])
\end{Verbatim}

Here, only one evaluation of the vector field per interval is needed as in one dimension the integration on the boundaries reduces to the evaluation of a single boundary point.  On a partition with $n$ intervals, this method then yields an accuracy of $\mathcal{O}(n^{-1})$.
Note that from Corollary~\ref{onedcorollary} it follows that the vector $v$ with $v_i = 1/F(x_i)$ is a left eigenvector of the
transition matrix~(\ref{ulam1dformula}) for the generator at the
eigenvalue 0. This fact proves pointwise convergence of the
invariant density of the discretization towards the real one.
Thus Ulam's method for the generator is very accurate and stable;  the $L^1$ error is $\mathcal{O}(1/n)$ as this is the rate at which approximants created from a basis of characteristic functions converge in $L^1$ to regular functions.

\paragraph{Spectral collocation}

\begin{Verbatim}[frame=single,fontsize=\footnotesize]
F = @(x) sin(4*pi*x)+1.1;                        % definition of the vector field
n = 32; Fx = F(0:1/n:1-1/n);                     % evaluation in collocation nodes
E = ifft(eye(n));                                % basis
FE = fft(diag(Fx)*E);                            % multiplication by vector field
I = [0:n/2-1 -n/2:-1];                           % frequencies
D = (2i*pi)*I';                                  % differentiation matrix
A = -D*ones(1,n).*FE;                            % discrete generator in frequency space
[v,lambda] = eig(A);                             % spectrum
f_n = @(x) real(exp(x'*I*2i*pi)*v(:,end));       % approximate density
L1 = quadl(f_n,0,1); fplot(@(x)f_n(x)/L1,[0,1])  % plot
\end{Verbatim}

Here, the vector field is evaluated once per grid point (cf.\ the second line of the code).  As predicted by Theorem~\ref{thm:spectral}, the accuracy increases exponentially with $n$ (cf.\ Figure~\ref{fig:sineflow2}).

\subsubsection{Computational efficiency}

In Figure~\ref{fig:sineflow2} (left) we compare the efficiency of the four methods in terms of how the $L^1$-error of the computed invariant density depends on the number of evaluations of the vector field.  The $L^1$-error was computed using an adaptive Lobatto quadrature as implemented in, e.g., Matlab's \texttt{quadl} command. These errors are due to the approximation space, the approximate numerical computation of matrix entries, and due to solving the discretized eigenvalue problem. To illustrate how these effects accumulate, in Figure~\ref{fig:sineflow2} (right) we compare the errors of three different approximations, for different (uniform) partitions: the approximate invariant densities obtained from the two Ulam type methods, and the projection of the true invariant density. We conclude that the error due to the approximation space dominates the total error.

\begin{figure}[h]
 \centering
     \includegraphics[width=0.50\textwidth]{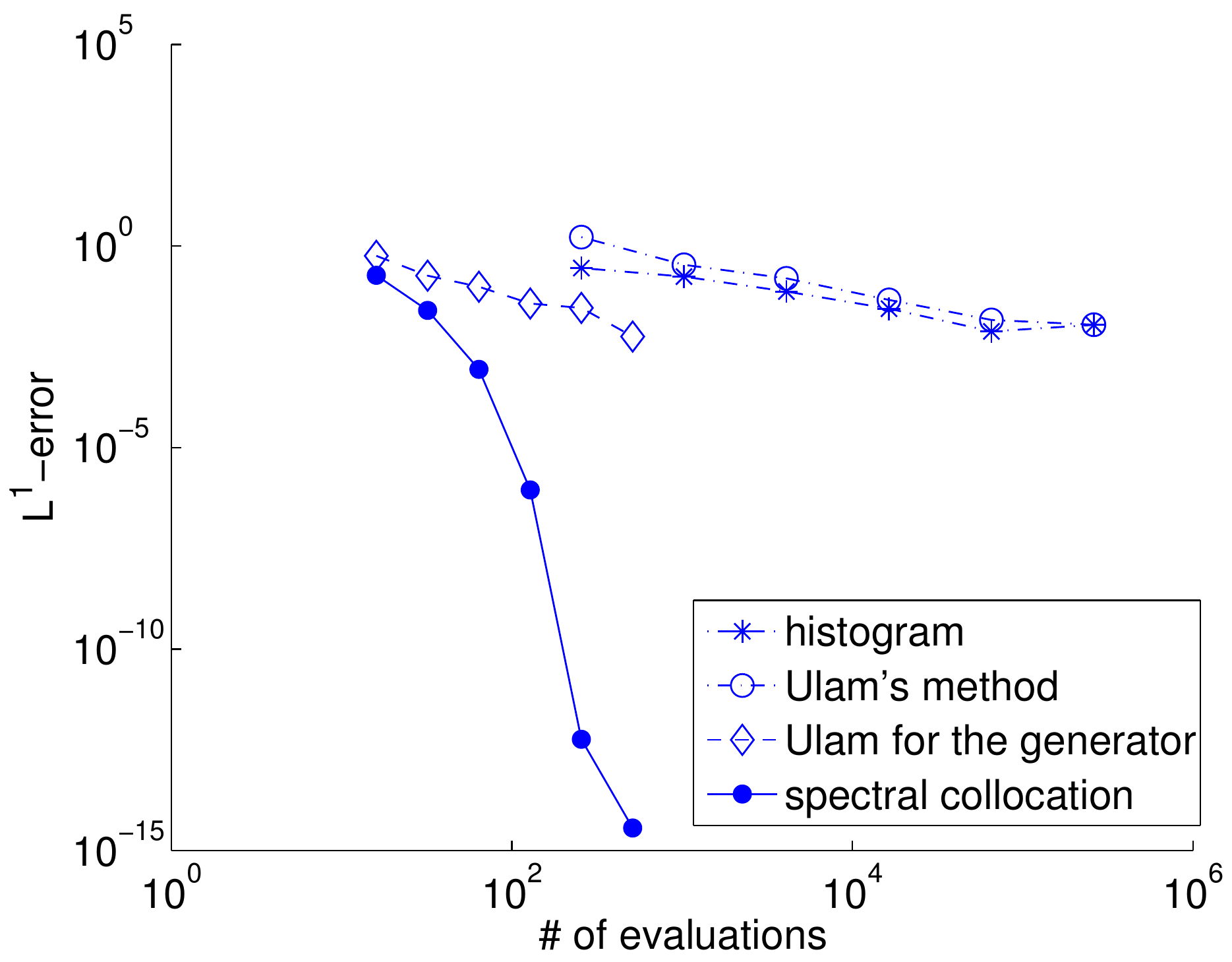}
     \includegraphics[width=0.49\textwidth]{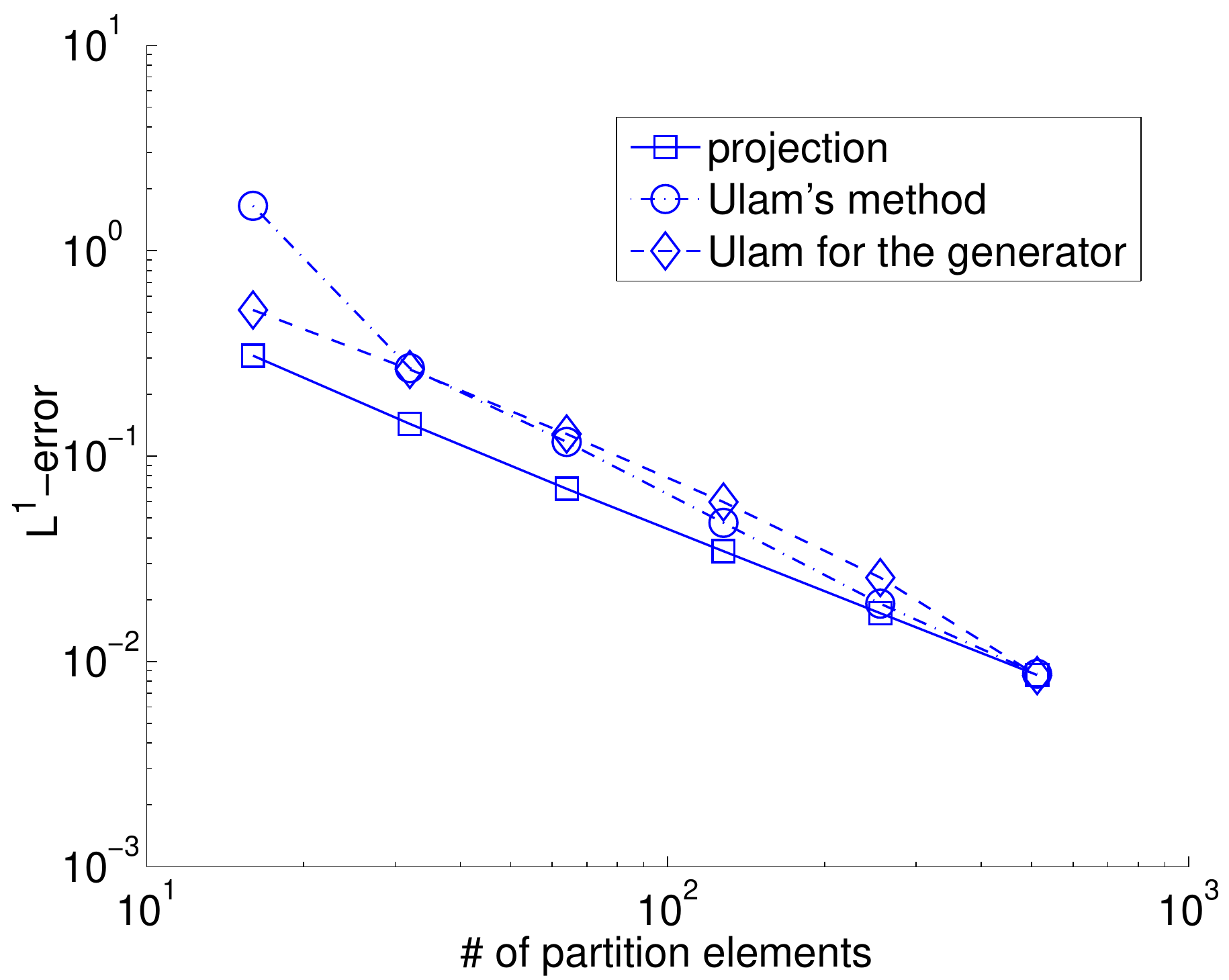}
 \caption{Left: $L^1$-error of the approximate invariant density as a function of the number of evaluations of the vector field ($2^4$ - $2^9$ partition elements, resp. collocation points, $2^8$ - $2^{18}$ number of iterations). Right: $L^1$-error of the approximate invariant densities obtained by projection of the true invariant density on the underlying partition, by Ulam's method, and by Ulam's method for the generator.}
 \label{fig:sineflow2}
 \end{figure}

\begin{table}[htb]
\centering\caption{Sources of computational cost. The dimension of the approximation space is denoted by $n$, $d$ denotes the dimension of state space.  In contrast to the other two methods, spectral collocation produces a full matrix.  One can reduce the cost for solving the eigenproblem in this case to $\mathcal{O}(n\log(n))\cdot\#(\text{GMRES iterations})$ by using an iterative solver (e.g.\ GMRES), and FFT to compute matrix-vector products.}
\renewcommand{\arraystretch}{1.5}
\begin{tabular}{l|c|c|c|c}
method & $\substack{\text{approximation}\\ \text{error}}$ & $\substack{\text{flops to set up matrix /}\\ \text{time integration}}$ & $\substack{\text{dimension} \\ \text{of spatial}\\ \text{integrals}}$ & $\substack{\text{flops to solve EVP /}\\ \text{vector iteration type}}$ \\[3pt] \hline
 Ulam's method & $\mathcal{O}(n^{-1/d})$ & $\mathcal{O}(n^{2})$ / yes & $d$ & $\mathcal{O}(n)$ / forw. \\[3pt]
$\substack{\text{Ulam's method}\\ \text{for the generator}}$ & $\mathcal{O}(n^{-1/d})$ & $\mathcal{O}\big(n^{(2d-1)/d}\big)$ / no & $d-1$ & $\mathcal{O}(n)$ / backw. \\[3pt]
$\substack{\text{spectral collocation}\\ \text{for the generator}}$ & $\mathcal{O}(n^{-k/d}),\ k\in\N$ & $\mathcal{O}(n)$ / no & $0$ & $\mathcal{O}(n^3)$ / backw.
\end{tabular}
\label{tab:error_sources}
\end{table}
The number of grid sets (32) chosen in Figure \ref{fig:sineflow} is tiny, and chosen merely for illustration purposes.  Likewise, the maximum number of evaluations used in both generator schemes of 512 is also obviously tiny, and in practice one could very cheaply increase the number of grid sets and concomitantly the number of evaluations.  The main message from this example is that in the right setting: low dimension, periodic domain, and infinitely smooth vector field, spectral collocation can significantly outperform standard Ulam and generator Ulam.
Ulam's method for the generator is clearly outperforming standard Ulam in this example, and we will see that it continues to do so across a variety of dimensions, domains, and vector fields.

\subsection{An area-preserving cylinder flow}

We consider an area-preserving flow on the cylinder, defined by interpolating a numerically given vector field as shown in Figure~\ref{fig:qg_vector_field}, which is a snapshot from a quasi-geostrophic flow, cf. \cite{TrPa94}.
The domain is periodic with respect to the $x$ coordinate and the field is zero at the boundaries $y=0$ and $y=8\cdot 10^{5}$.  Again, we apply the three methods discussed in Section~\ref{sec:methods} in order to compute approximate eigenfunctions of the transfer operator and the generator.
\begin{figure}[htb]
 \centering
 \includegraphics[width=0.4\textwidth]{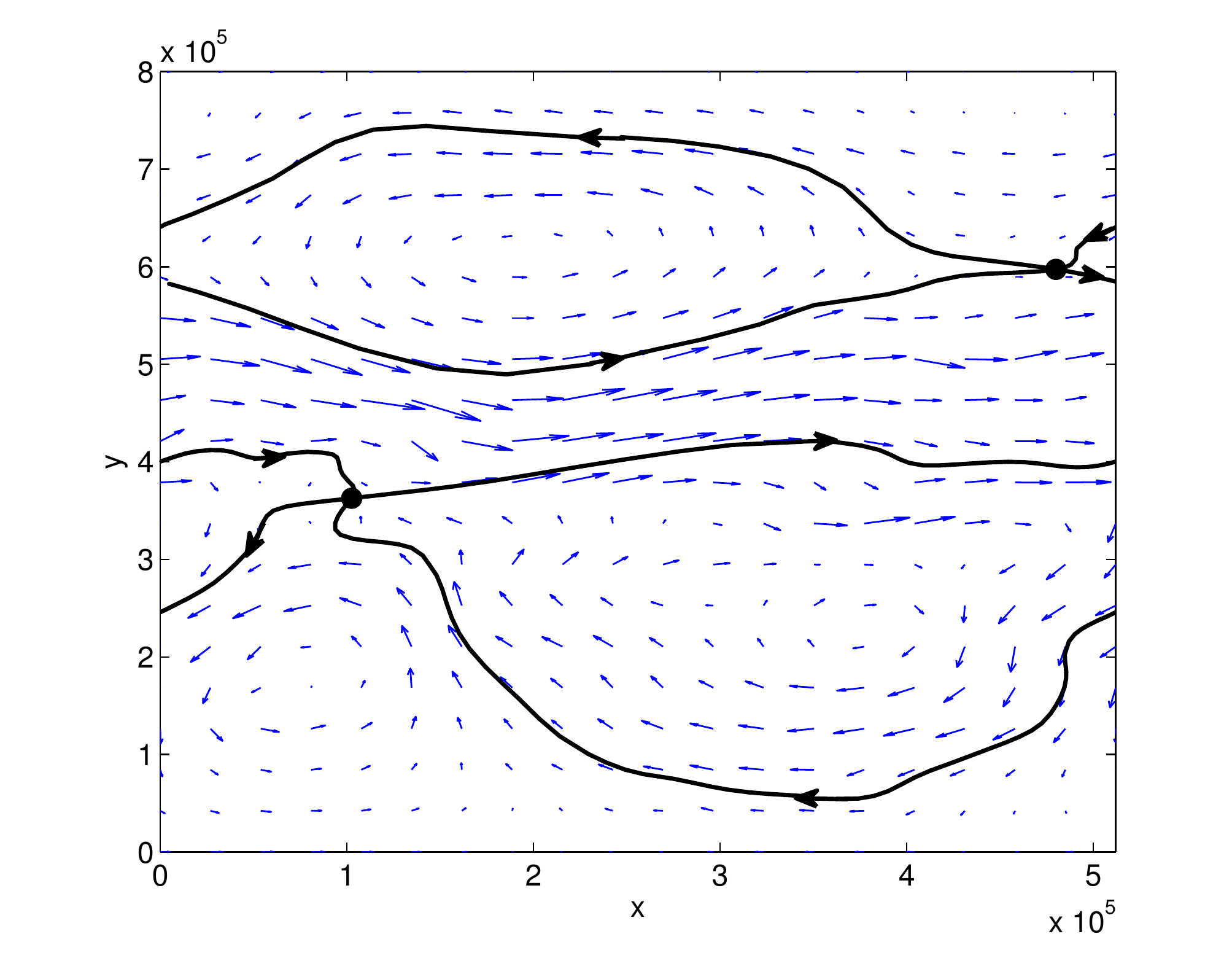}
 \caption{Vector field of the area-preserving cylinder flow.  Shown are also two saddle equilibria (black dots) together with their stable and unstable manifolds.}
 \label{fig:qg_vector_field}
\end{figure}

\subsubsection{Perturbing the model}   Since the flow is area-preserving, we have a continuum of closed orbits, and each of them is an invariant set.  On each, an invariant measure can be supported, i.e.\ the eigenspace at $1$ (resp.\ $0$) is infinite-dimensional.  Both Ulam type methods can be interpreted as a small random perturbation of the original system (albeit different ones), yielding a unique stationary eigenvector of the resulting matrix. However, simultaneously with decreasing the cell size, the size of this ``built-in'' perturbation shrinks and it is therefore not immediate to which invariant measure the discretization schemes converge to. In order to make the problem well posed and the results comparable among the three methods we therefore need to artificially add a random perturbation (in form of a diffusion term) to the model.

We therefore choose the noise level $\ep$ such that the resulting diffusion coefficient $\ep^2/2$ has the same order of magnitude as the numerical diffusion present within Ulam's method for the generator, but is still larger than the latter.  Since the estimate from Remark~\ref{rem:upwind} yields a numerical diffusion coefficient of $\approx 120$ (see also~\cite{Ko10a}), we choose $\ep=\sqrt{2\cdot 500}$ here.

In \textbf{Ulam's method}, for the simulation of the SDE (\ref{eq:SDE}) a fourth order Runge-Kutta method is used, where in every time step a properly scaled (by a factor $\sqrt{h}\cdot\ep$, where $h$ is the time step) normally distributed random number is added.  We use $1000$ sample points per box and the integration time $T = 5\cdot 10^6$, which is realized by 20 steps of the Runge--Kutta method.  Note that the integrator does not know that the flow lines should not cross the lower and upper boundaries of the state space. Points that leave phase space are projected back along the $y$ axis into the next boundary box. An adaptive step--size control could resolve this problem, however at the cost of even more right hand side evaluations.  The domain is partitioned into $
128 \times 128$ boxes.
For \textbf{Ulam's method for the generator}, again, we employ a partition of $128 \times 128$ boxes and approximate the edge integrals by the trapezoidal rule using three nodes (which turns out to be sufficiently accurate in this case).
For \textbf{spectral collocation}, we employ 51 Fourier modes in the $x$ coordinate (periodic boundary conditions) and the first 51 Chebyshev polynomials in the $y$ coordinate, together with Neumann boundary conditions.  
Due to spectral convergence, this smaller number of modes already yields an accuracy in the invariant density comparable to the one of the other two approaches.

\subsubsection{Computing almost invariant sets}
In Table~\ref{tab:qg-effort} we compare approximations for the three leading eigenvalues from our three schemes.  If we assume that spectral collocation gives a very accurate result, Ulam's method should be $\mathcal{O}(n^{-1})$ from it. With $n=128$ this matches well with the numbers in the table.  Ulam's method for the generator generates additional numerical diffusion; this scales the corresponding eigenvalues.
\begin{table}[htb]
\caption{Approximate eigenvalues.}
\centering
\begin{tabular}{l|c|c|c}
   method																	& $\lambda_2$ & $\lambda_3$ & $\lambda_4$ \\\hline
   Ulam's method ($\log(\lambda_i)/T$)	& $-1.64\cdot 10^{-8}$ & $-0.91\cdot 10^{-7}$ & $-1.06\cdot 10^{-7}$ \\
   Ulam's method for the generator		    & $-1.98\cdot 10^{-8}$ & $-1.03\cdot 10^{-7}$ & $-1.19\cdot 10^{-7}$ \\
   spectral collocation	for the generator			& $-1.65\cdot 10^{-8}$ & $-0.92\cdot 10^{-7}$ & $-1.05\cdot 10^{-7}$ 
\end{tabular}
\label{tab:qg-effort}
\end{table}
In Figure~\ref{fig:qg-evs} we compare the approximate eigenvectors
of the second, third and fourth relevant eigenvalues of the transfer
operator (resp.\ generator).   Clearly, they all give the
same qualitative picture.  Yet, the number of evaluations of the
vector field (and thus the associated cpu times) differ significantly, as shown in Table~\ref{tab:qg-effort2}.
% Requires the booktabs if the memoir class is not being used
%
\begin{table}[htbp]
\centering\caption{Number of evaluations of the vector field and cpu times (on a 2.5 GHz Core Duo) for computing the approximate operator resp.\ generator and cpu times for computing the leading 4 eigenvalues/-vectors.}
\begin{tabular}{l|l|l|l}
   method    & \# of rhs evals & time matrix & time \tt{eigs}\\\hline
   Ulam's method                   & $\approx 3\cdot 10^8$ & 3020 sec. & 1.0 sec.\\
   Ulam's method for the generator      & $\approx 3\cdot 10^5$ & 143 sec. & 0.8 sec.\\
   spectral collocation for the generator    & $\approx 3\cdot 10^3$ & 0.4 sec. & 3 sec.
\end{tabular}
\label{tab:qg-effort2}
\end{table}
\begin{figure}
 \centering
     \subfigure[Ulam's method]{
         \includegraphics[width=0.32\textwidth]{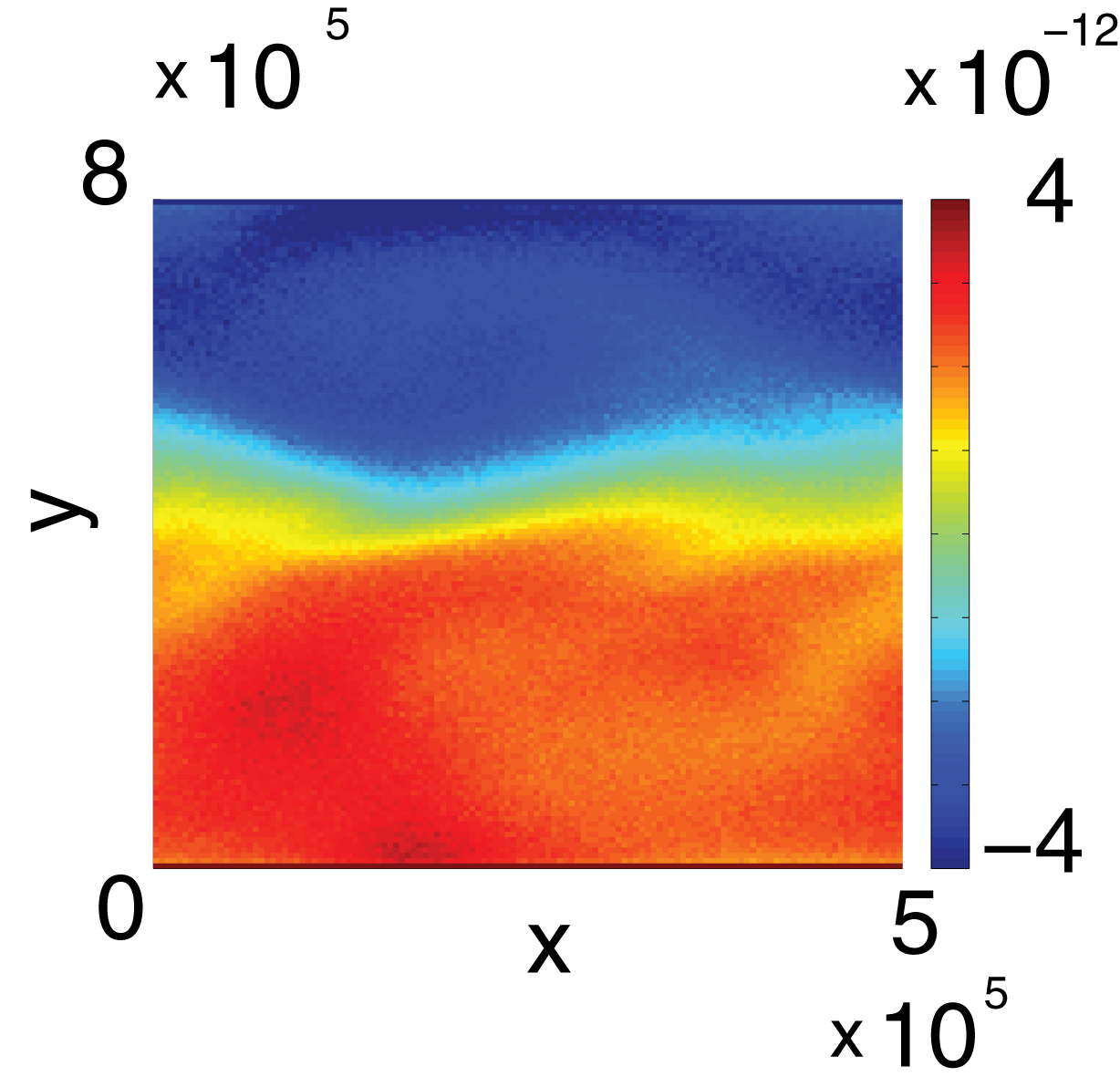}
         \includegraphics[width=0.32\textwidth]{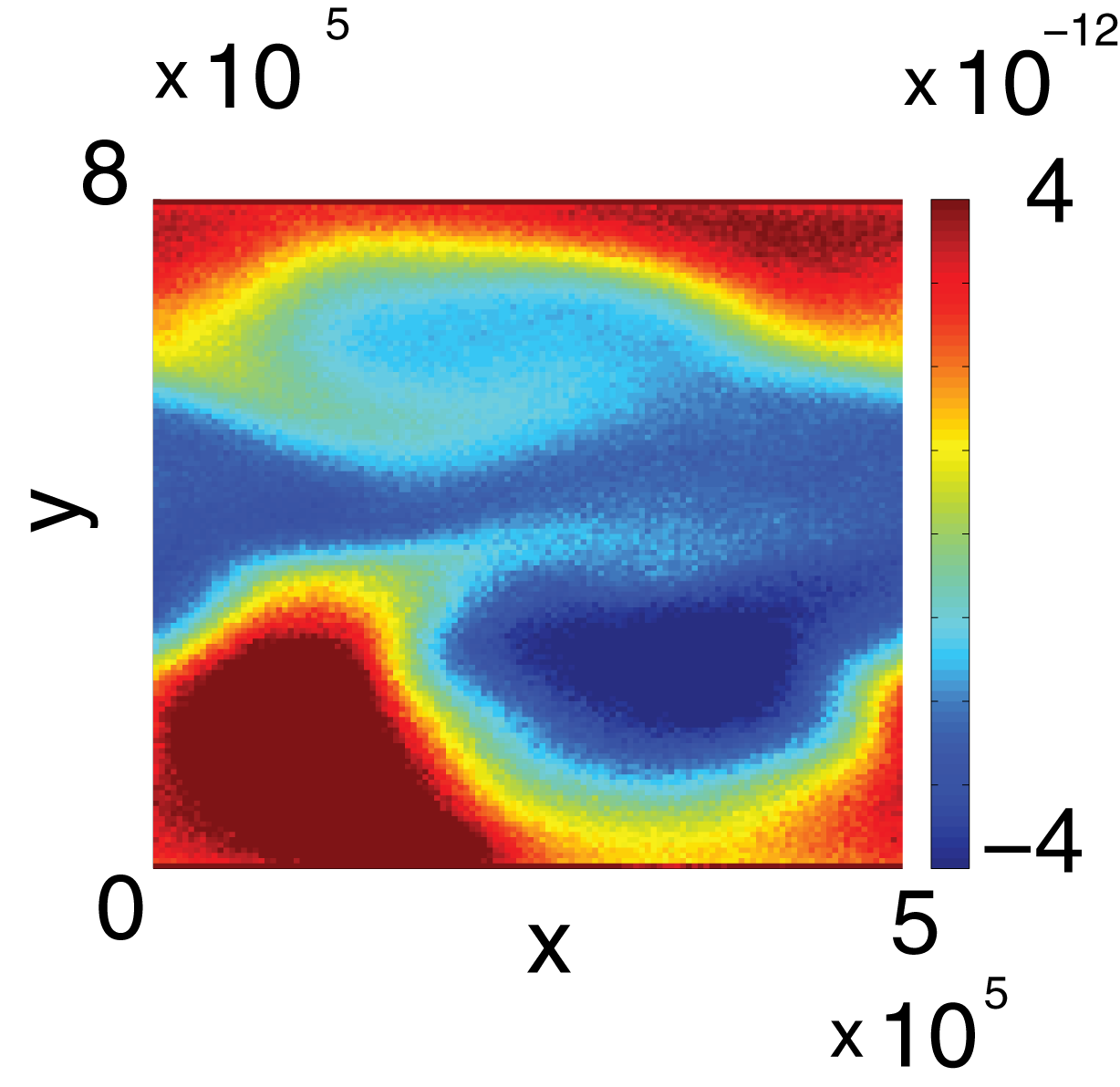}
         \includegraphics[width=0.32\textwidth]{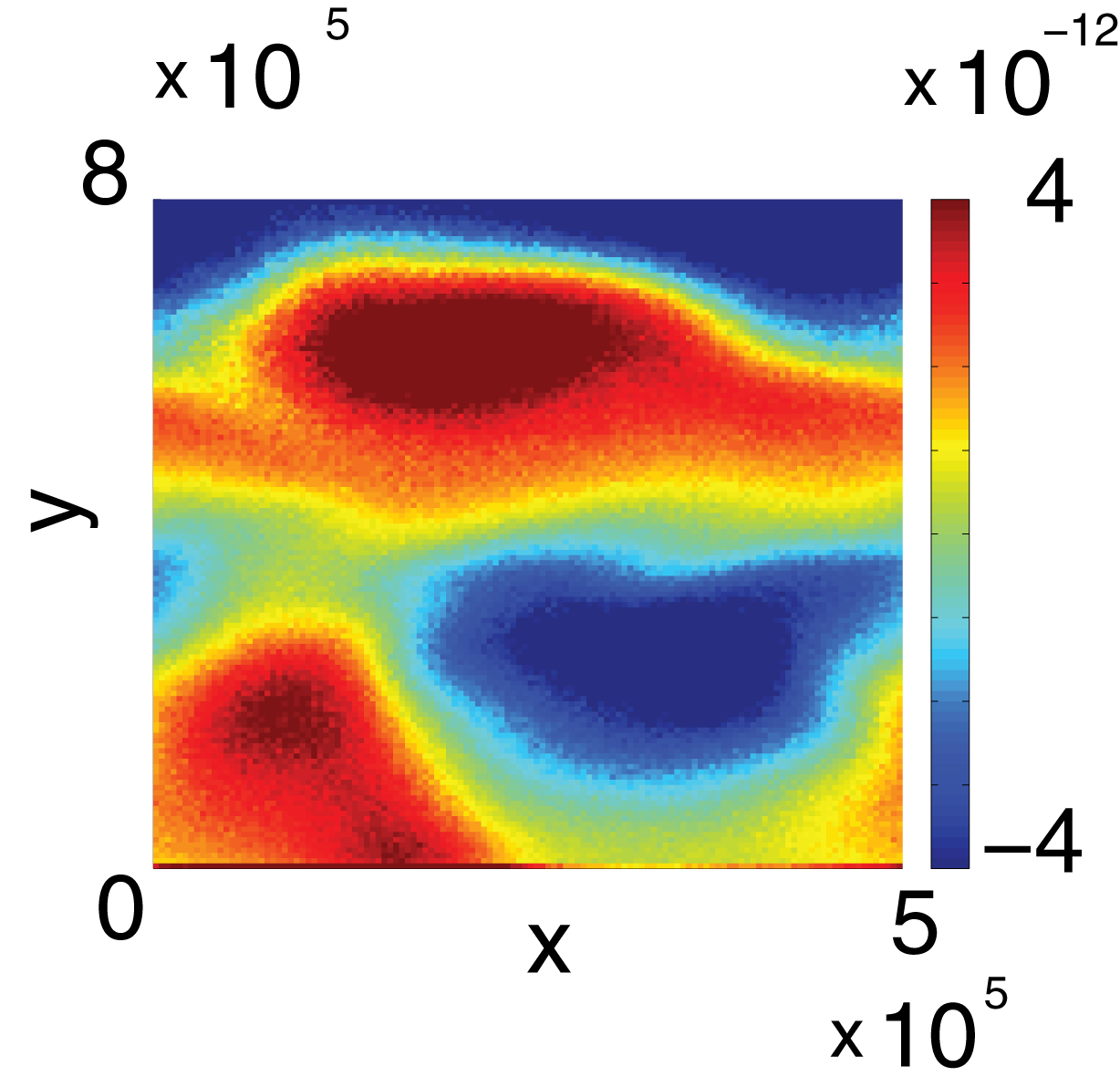}
         \label{fig:qg_ulam}}
     \subfigure[Ulam's method for the generator]{
         \includegraphics[width=0.32\textwidth]{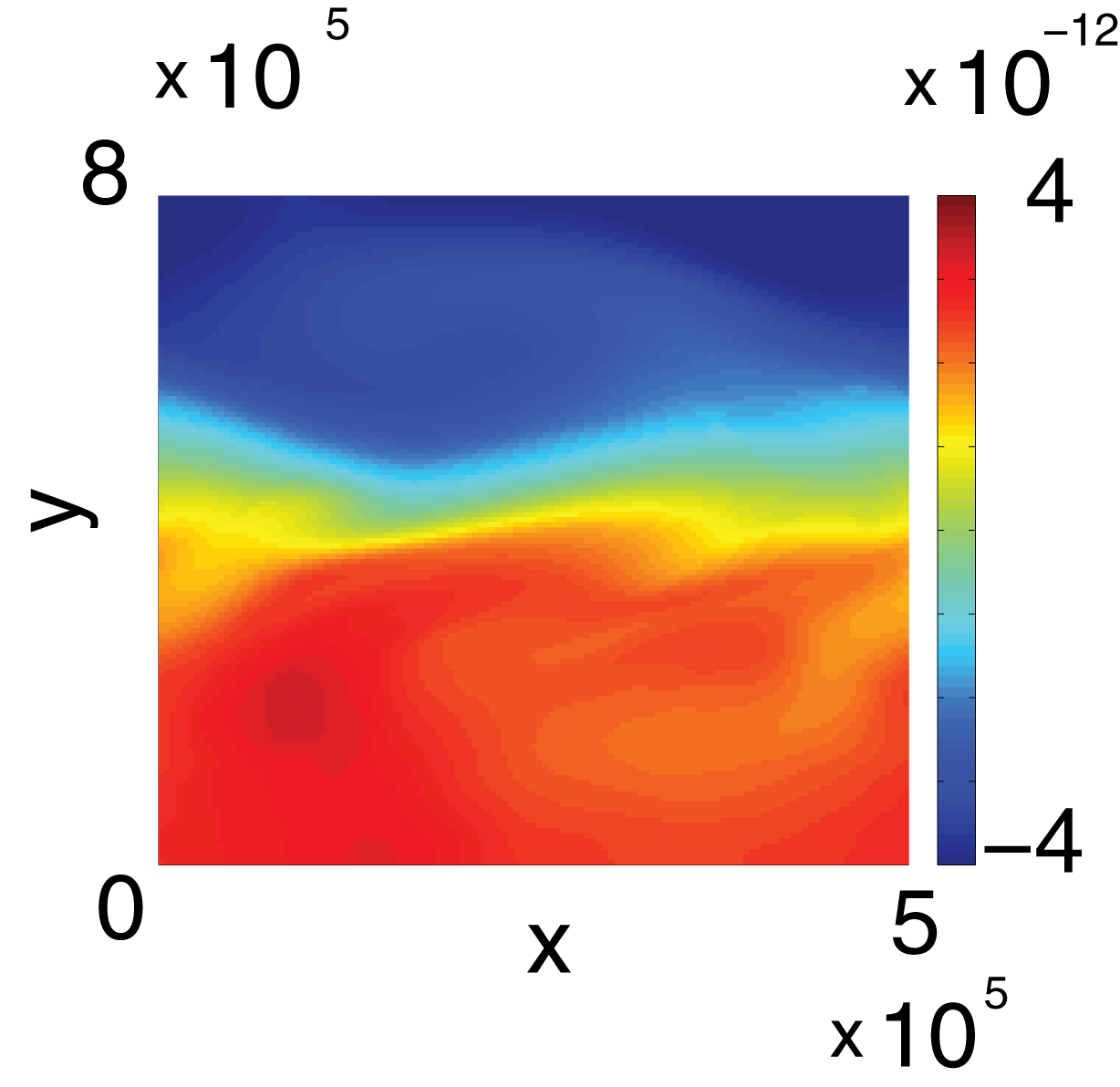}
         \includegraphics[width=0.32\textwidth]{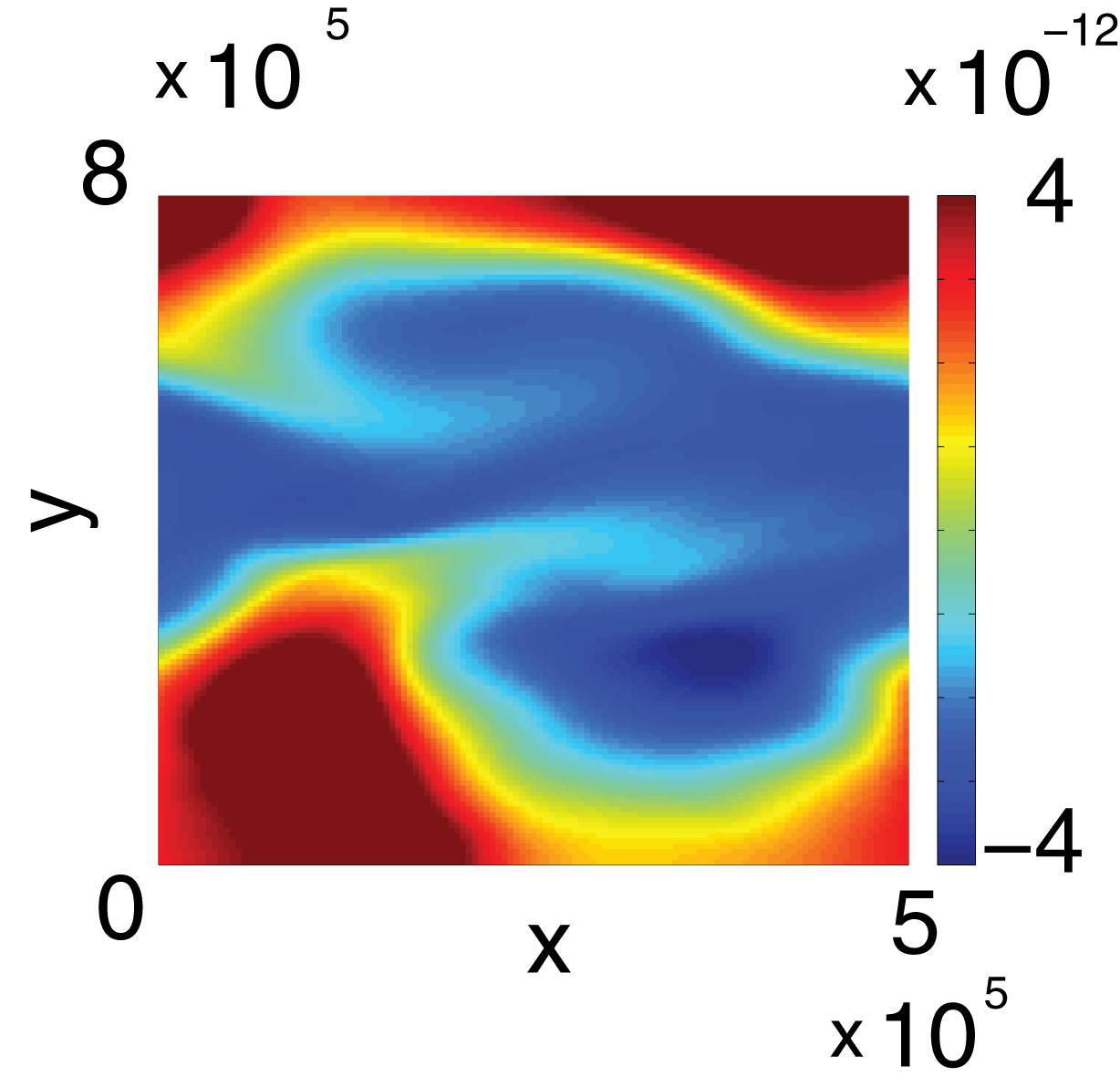}
         \includegraphics[width=0.32\textwidth]{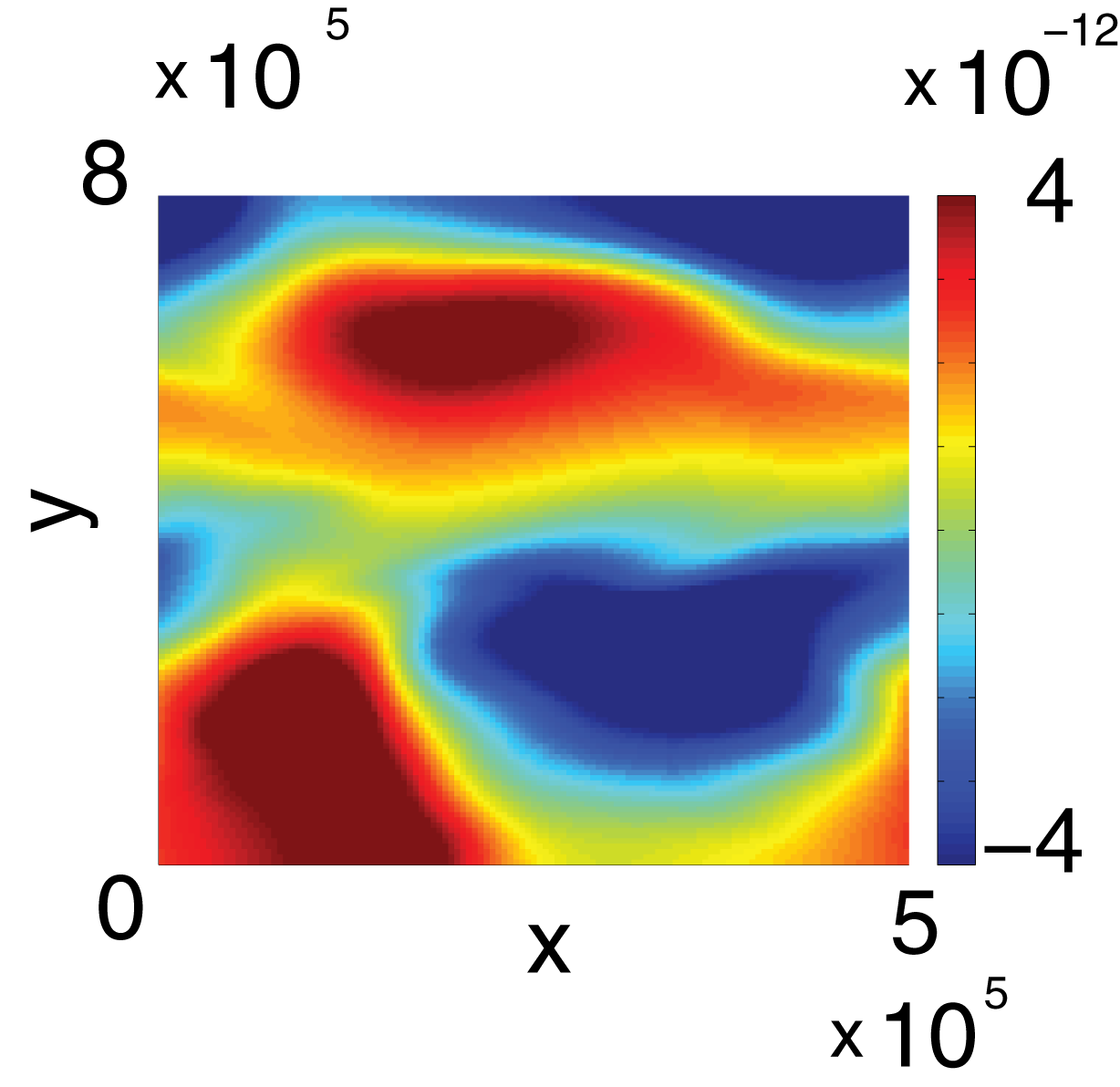}
         \label{fig:qg_ulam_ig}}
     \subfigure[Spectral collocation for the generator]{
         \includegraphics[width=0.32\textwidth]{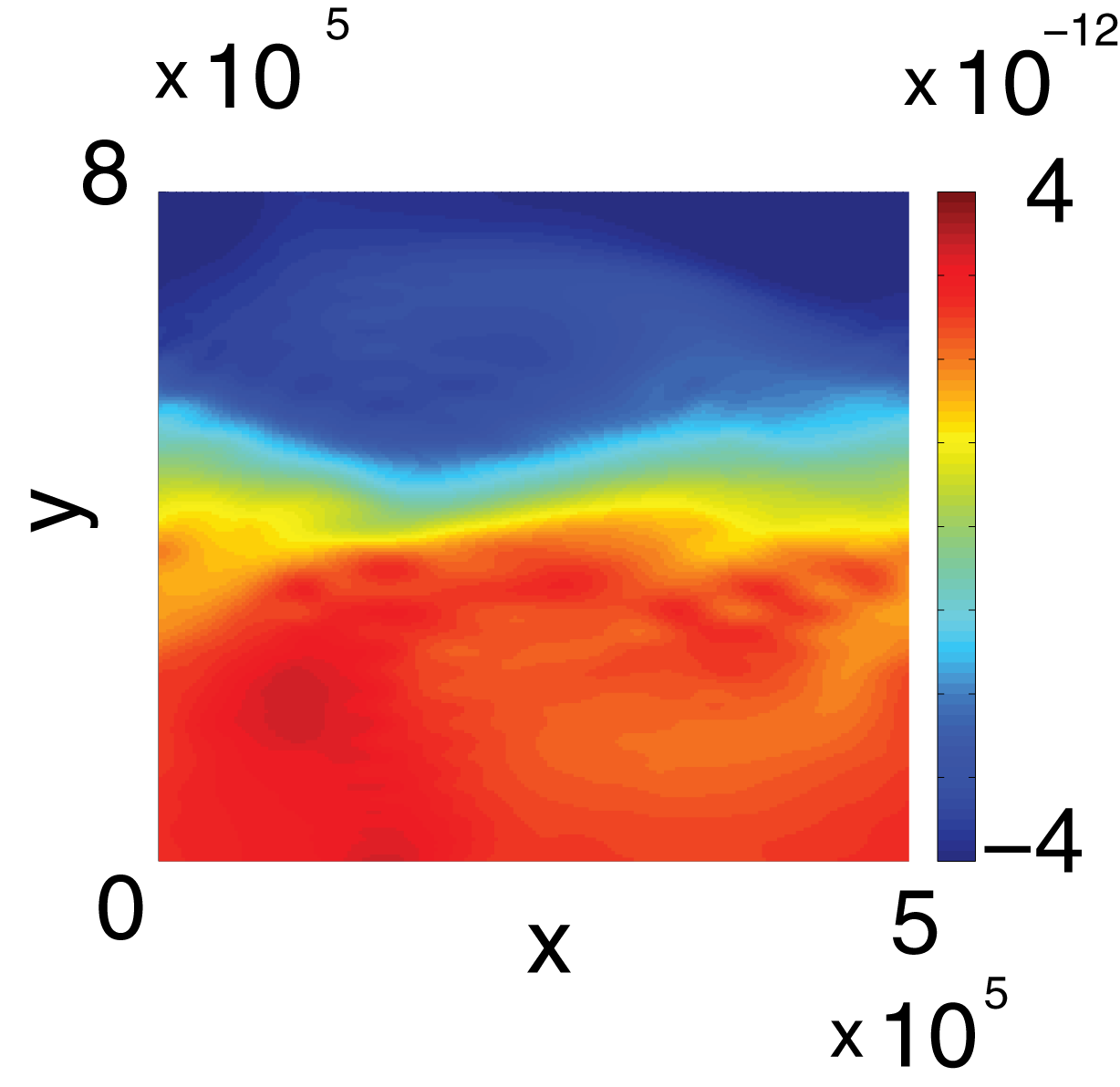}
         \includegraphics[width=0.32\textwidth]{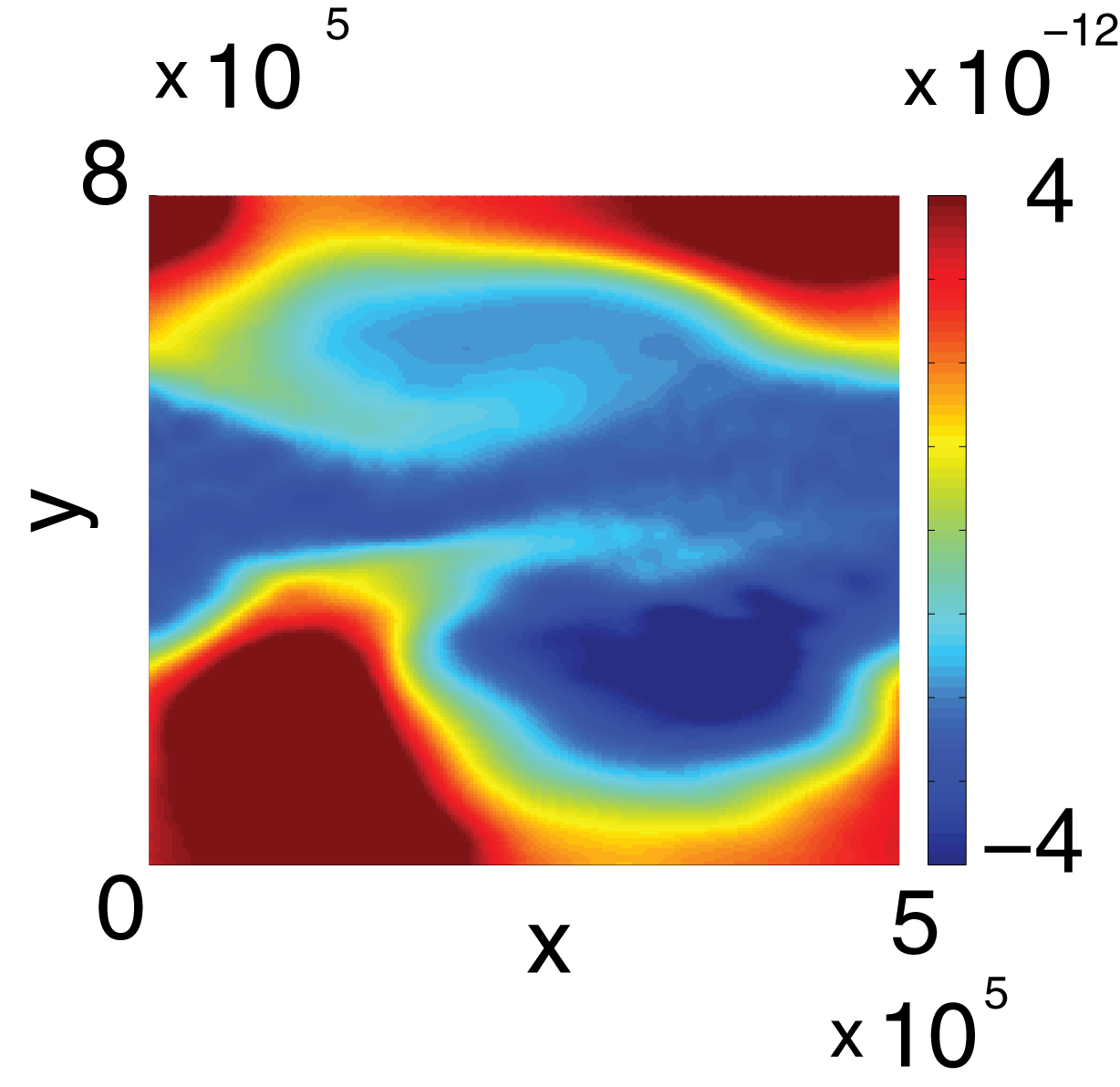}
         \includegraphics[width=0.32\textwidth]{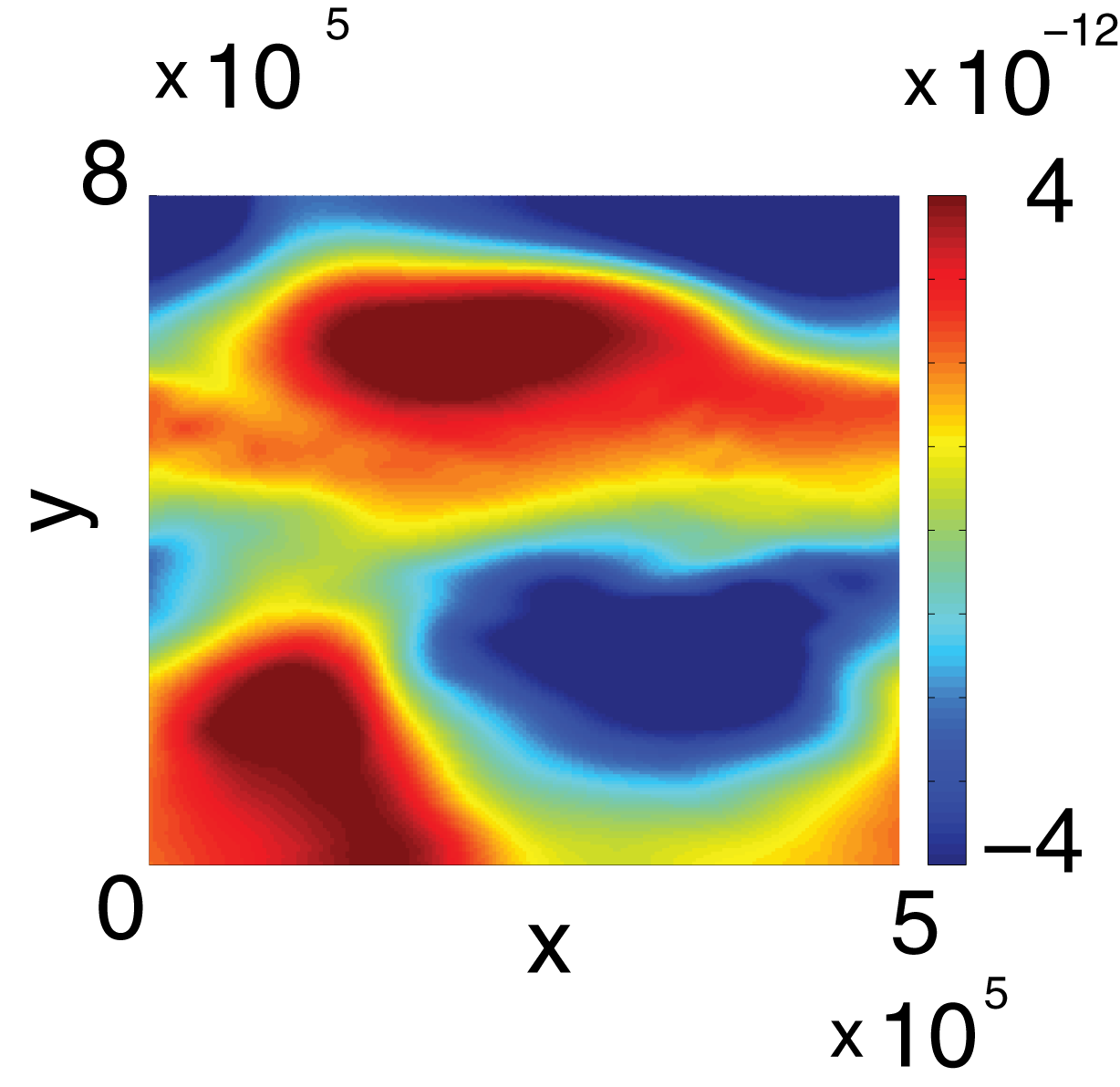}
         \label{fig:qg_spectral}}
     \caption{Quasi-geostrophic flow: Eigenvectors at the second, third and fourth eigenvalue (from left to right).}
     \label{fig:qg-evs}
\end{figure}
\begin{figure}[htb]
	\centering
   \includegraphics[width=0.42\textwidth]{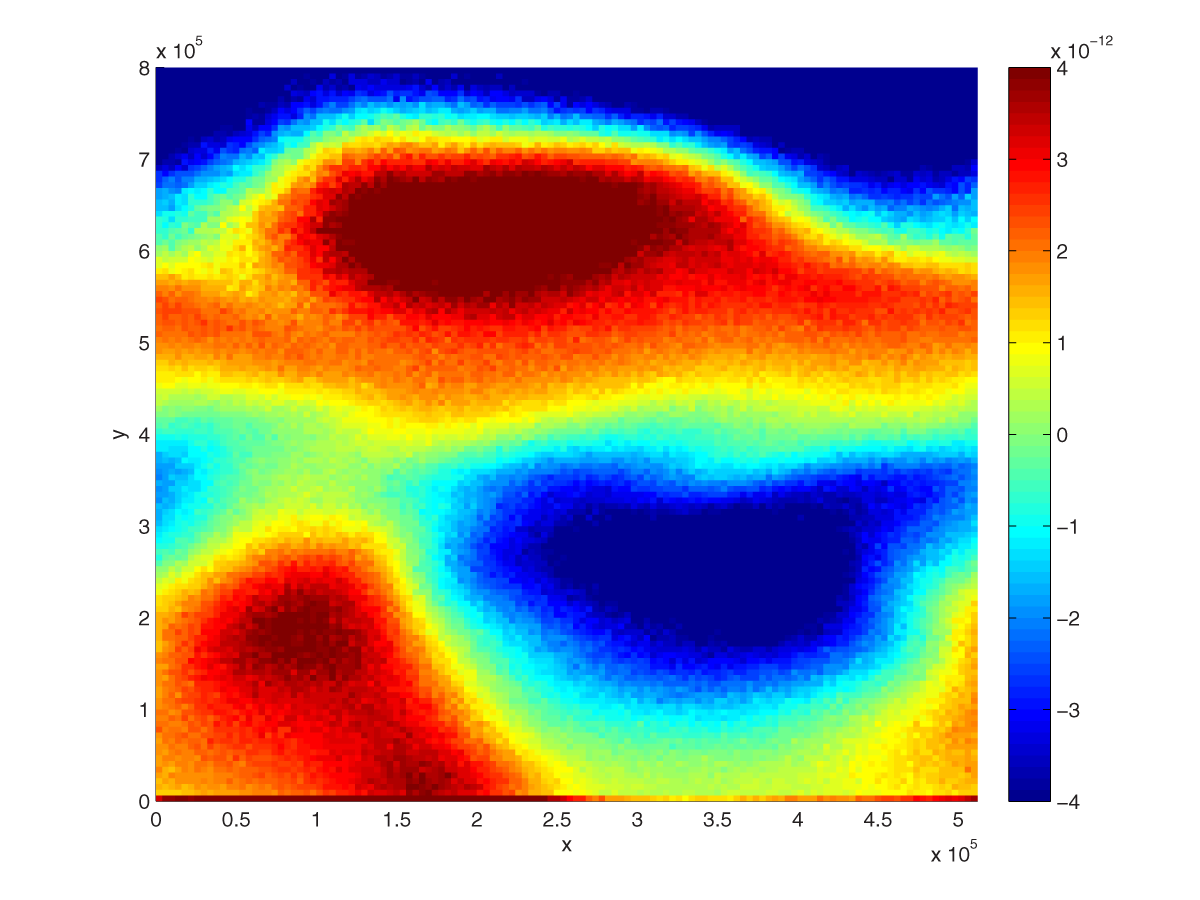}\qquad
	\includegraphics[width=0.42\textwidth]{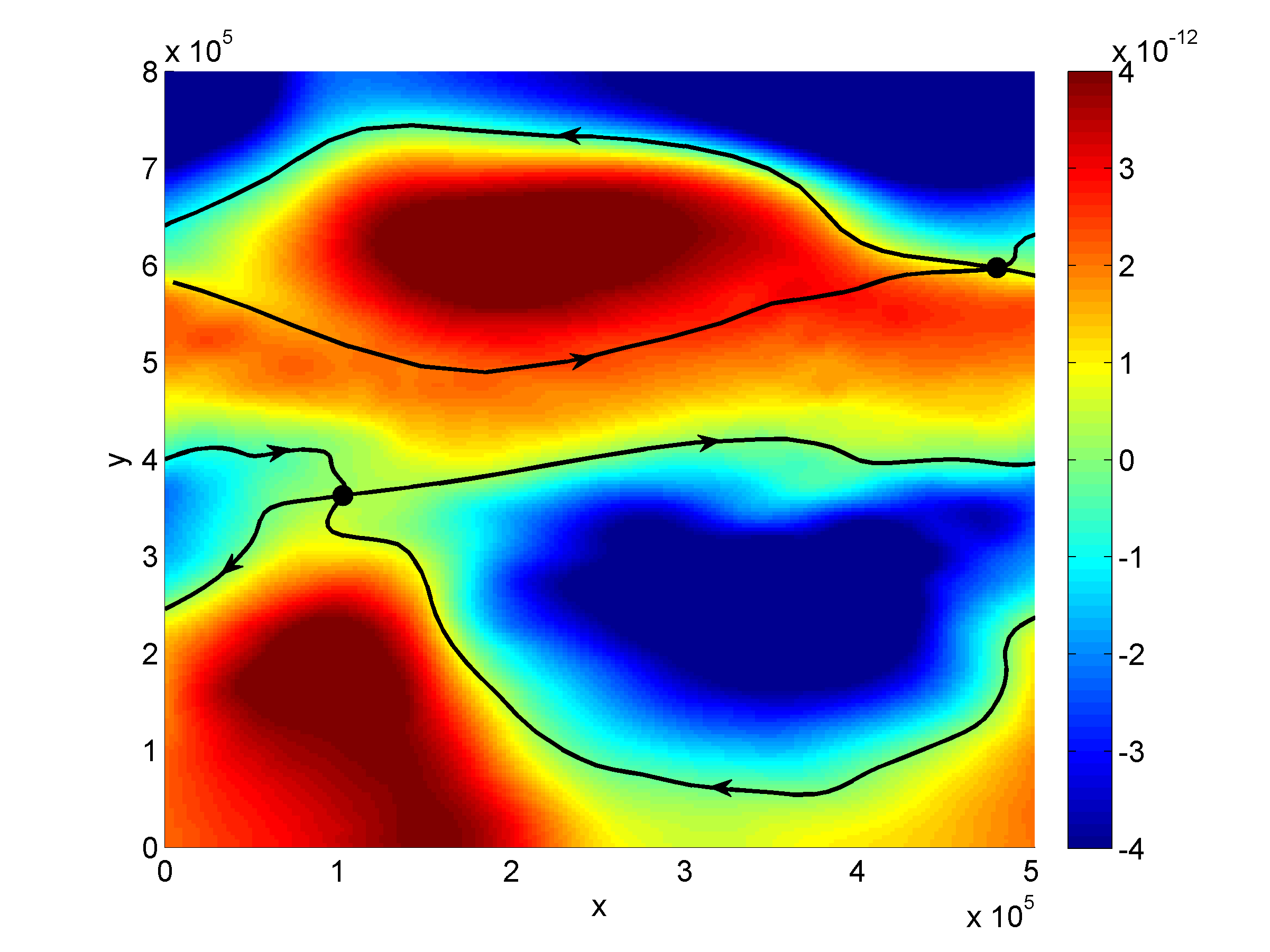}
	\caption{Comparison of the approximate eigenvector at $\lambda_4$ by Ulam's method (left) with the one computed by Ulam's method for the generator (right, together with the two saddle equilibria and their invariant manifolds).}
		\label{fig:qgseporbits}
\end{figure}
The regions enclosed by the stable/unstable manifolds of the two saddle equilibria shown in Figure \ref{fig:qg_vector_field} are obviously invariant for the noiseless flow. The introduction of noise means that it is possible (but with a rather low probability) for stochastic trajectories to jump between these regions, making these regions almost-invariant sets (cf.\ Section~\ref{subsec:almost}).  Figure \ref{fig:qg-evs} shows that these almost-invariant sets are highlighted in the 2nd, 3rd, and 4th eigenvectors as level sets at high gradient regions in the eigenvectors (in this example, near the yellow bands). The 2nd and 3rd eigenvectors highlight regions near different stable/unstable manifolds at the yellow bands and the 4th eigenvector shown in Figure~\ref{fig:qgseporbits} shows a combination. We refer the reader to \cite{DeJuKo05a,FrPa09a} for details on the connection between invariant manifolds and almost invariant sets.  As one sees in Figure~\ref{fig:qgseporbits}, the eigenvectors from Ulam's method for the generator are smoother compared to the ones from the standard Ulam method.  Ulam's method for the generator appears to be more accurate, at least for this system with a small amount of noise.  A possible explanation is that with Ulam's method for the generator we formally average the diffusion, while with the standard Ulam method we simulate a finite number of random trajectories, which incompletely sample the true noise distribution.

\subsection{A volume-preserving 3D example: the ABC-flow}		\label{ssec:ABCflow}

In this section we demonstrate the efficacy of our generator approach for a three-dimensional flow.
We consider the so-called ABC-flow \cite{Ar65a}, given by
\begin{eqnarray*}
 \dot{x} & = & a\sin(2\pi z)+c\cos(2\pi y) \\
 \dot{y} & = & b\sin(2\pi x)+a\cos(2\pi z) \\
 \dot{z} & = & c\sin(2\pi y)+b\cos(2\pi x),
\end{eqnarray*}
on the 3 dimensional torus. The flow is volume-preserving, so in order to make the computation of the spectrum well posed we need to add some artifical diffusion again. The vector field is very smooth, and in numerical experiments a diffusion coefficient $\varepsilon = 0.04$ turned out to work well. For $a=\sqrt{3}$, $b=\sqrt{2}$ and $c=1$, there is numerical evidence that the flow exhibits complicated dynamics and invariant sets of complicated geometry \cite{DoFrHe86a,
FrPa09a}.

We approximate the leading few eigenfunctions of $\A_\ep$ with Ulam's method for the generator on a $64\times 64\times 64$ grid.  We used Gauss quadrature with 16 nodes on each face of a box in
order to approximate the surface integrals, requiring $\sim 16\cdot 64^3 \approx 4\cdot 10^6$ evaluations of the vector field.  The $L^1$-error of the approximate invariant density is $\approx 10^{-12}$.
The second and fifth eigenfunctions are shown in Figure \ref{fig:ABC_U64_LEV2,5}. The regions highlighted in red and blue are invariant cylinders of the deterministic flow;  see \cite{FrPa09a} for more details on how to extract invariant and almost-invariant sets from the eigenfunctions.
\begin{figure}[ht]
 \centering
     \includegraphics[width=0.9\textwidth]{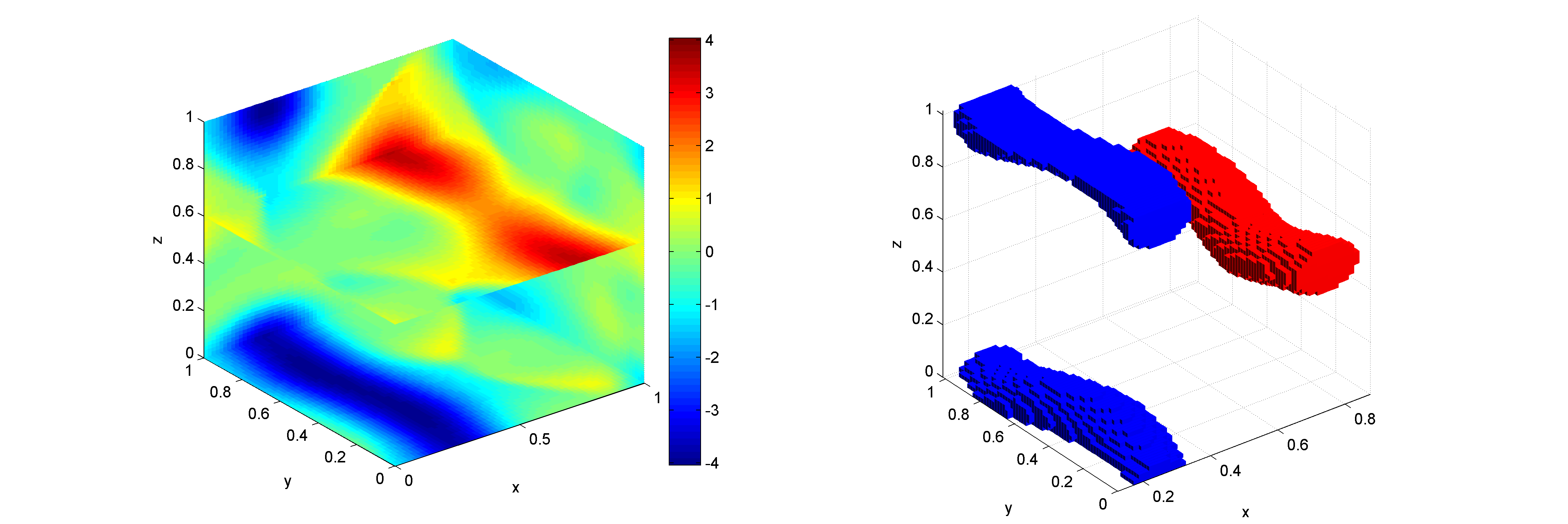}\\
     \includegraphics[width=0.9\textwidth]{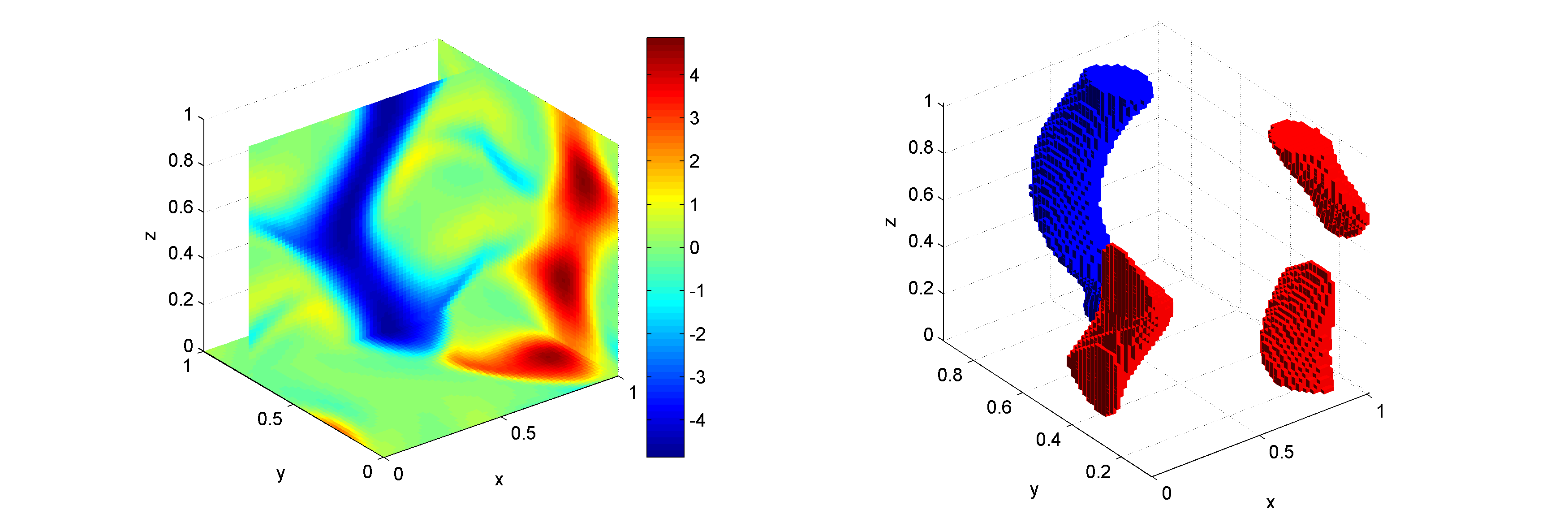}
 \caption{Second and fifth eigenvectors of the infinitesimal generator -- Ulam type discretization. Left: slices, right: regions, where the absolute value of the left eigenfunctions exceed a given threshold -- the ``core" of the almost invariant sets, representing invariant cylinders in this example.}
 \label{fig:ABC_U64_LEV2,5}
\end{figure}

We also apply spectral collocation on a $11\times11\times11$ grid, requiring $11^3 = 1331$ evaluations of the vector field.
The $L^1$-error of the approximate invariant density is
$\approx 10^{-14}$.  We observe fast convergence (after
comparing with computations on coarser resolutions) of not just the
first nontrivial eigenfunctions, but also of higher order ones.
\begin{figure}[h]
 \centering
     \includegraphics[width=0.9\textwidth]{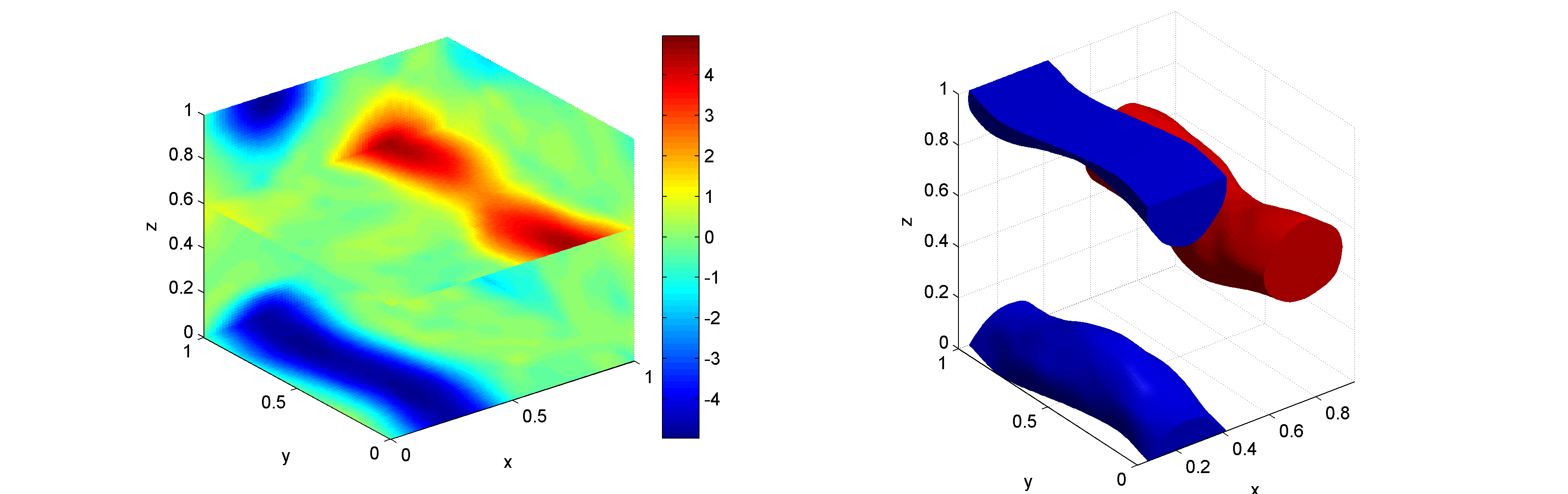}\\
     \includegraphics[width=0.9\textwidth]{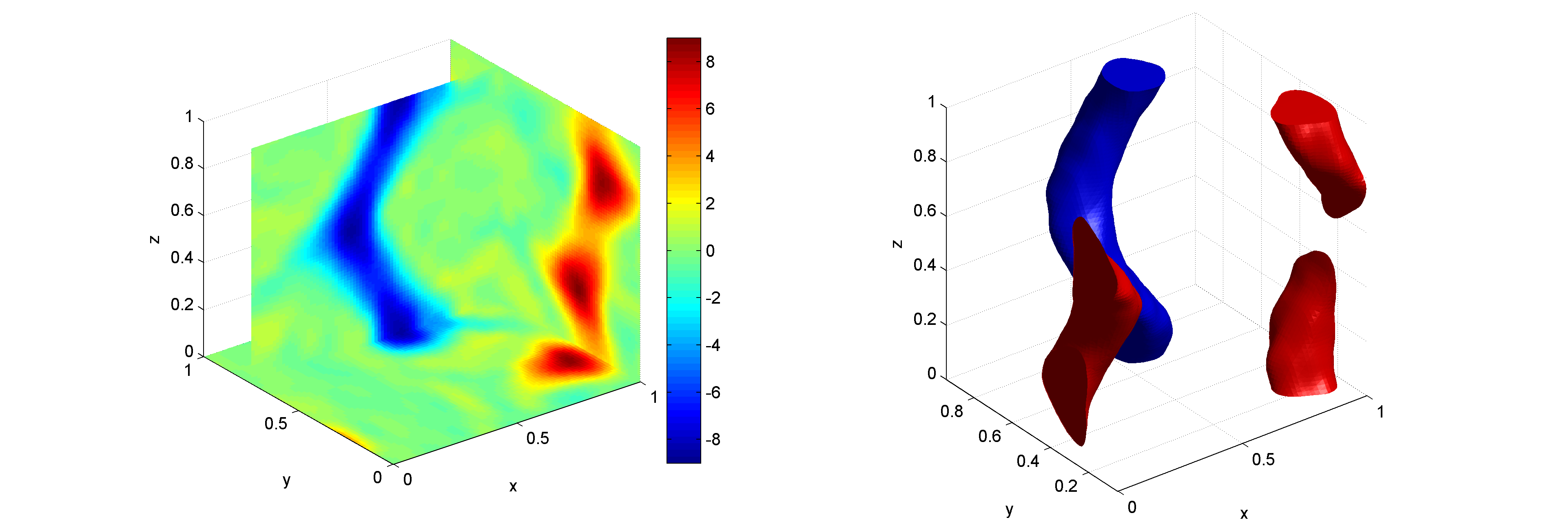}
 \caption{Second and fifth eigenfunction of the generator computed by spectral collocation. Left: slices, right: regions, where the absolute value of the left eigenfunctions exceed a given threshold -- the ``core" of the almost invariant sets.}
 \label{fig:ABC_res13_EV2,6}
\end{figure}

\subsection{A dissipative 3D example with complicated geometry: the Lorenz system}

As a last example we consider a system where the effective dynamics
is supported on a set of complicated geometry and fractional
dimension, the well-known Lorenz system
\begin{align*}
\dot x & =  \sigma(y-x), \\
\dot y & =  x(\varrho-z)-y, \\
\dot z & =  xy-\beta z,
\end{align*}
where we make the standard parameter choices $\sigma=10$, $\varrho=28$ and $\beta=8/3$. Since the effective (complicated) dynamics happens on a lower dimensional attractor we do not expect spectral collocation to work well and we therefore apply Ulam's method for the generator here.

A decade ago, numerical techniques were introduced for
computing box coverings for attractors of complicated structure, cf.
\cite{DeHo97a,DeFrJu01a}. These techniques make use of the fact
that the set $M$ to be computed is an attractor, hence each
trajectory starting in its vicinity will be pulled to $M$ in a
fairly short time. In our approach time is not considered, we only
use the vector field. Since the boundary of the box covering does
not have to coincide with the boundary of the attractor, a tight box
covering might not show the desired results, because of relatively big
outflow rates in boundary boxes. The simplest idea is to use a
big enough rectangle -- in our case $[-30,30]\times [-30,30]\times
[-10,70]$. However, this set is not invariant, hence the question
arises, which boundary conditions one should impose. Translating the
condition ``no probability flow on the boundary'' in this Ulam type
approach leads to the condition that the flow rates at the boundary
are ignored. The attractor is then extracted by simple thresholding:
regions where the approximate invariant density is nonzero are
expected to lie in a small neighbourhood of the attractor.  The
finer the resolution, the smaller the diffusion introduced by the
discretization and thus the tighter this neighbourhood of the
attractor is.  Having constructed the attractor, we may restrict the
other eigenfunctions to this set in order to obtain almost
invariant sets within the attractor.

For our computation, we employed a grid of $128\times 128\times 128$ boxes and used Gauss quadrature in $5 \times 5$ nodes on the faces in order to set up the approximate generator.  The attractor is extracted by thresholding the approximate invariant density $u$: We cut off at a threshold value $5\cdot 10^{-6}$ such that 96\% of the invariant measure is supported on $\{u_1 > c\}$.
Figure~\ref{fig:LorAttr128} shows the corresponding covering as well as the sign structure of the second and third eigenvector.

\begin{figure}[h]
 \centering
     \includegraphics[width=0.32\textwidth]{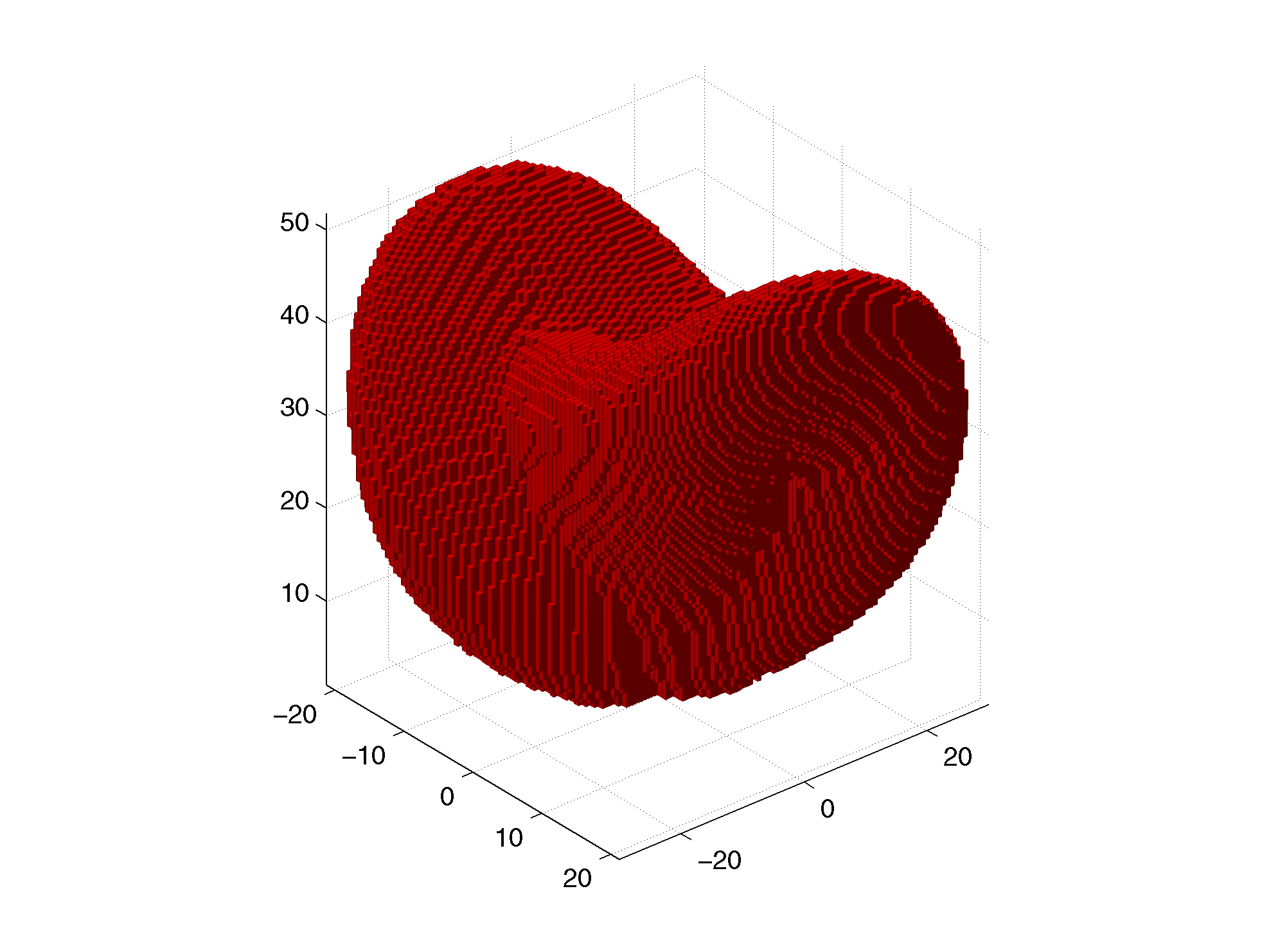}
     \includegraphics[width=0.32\textwidth]{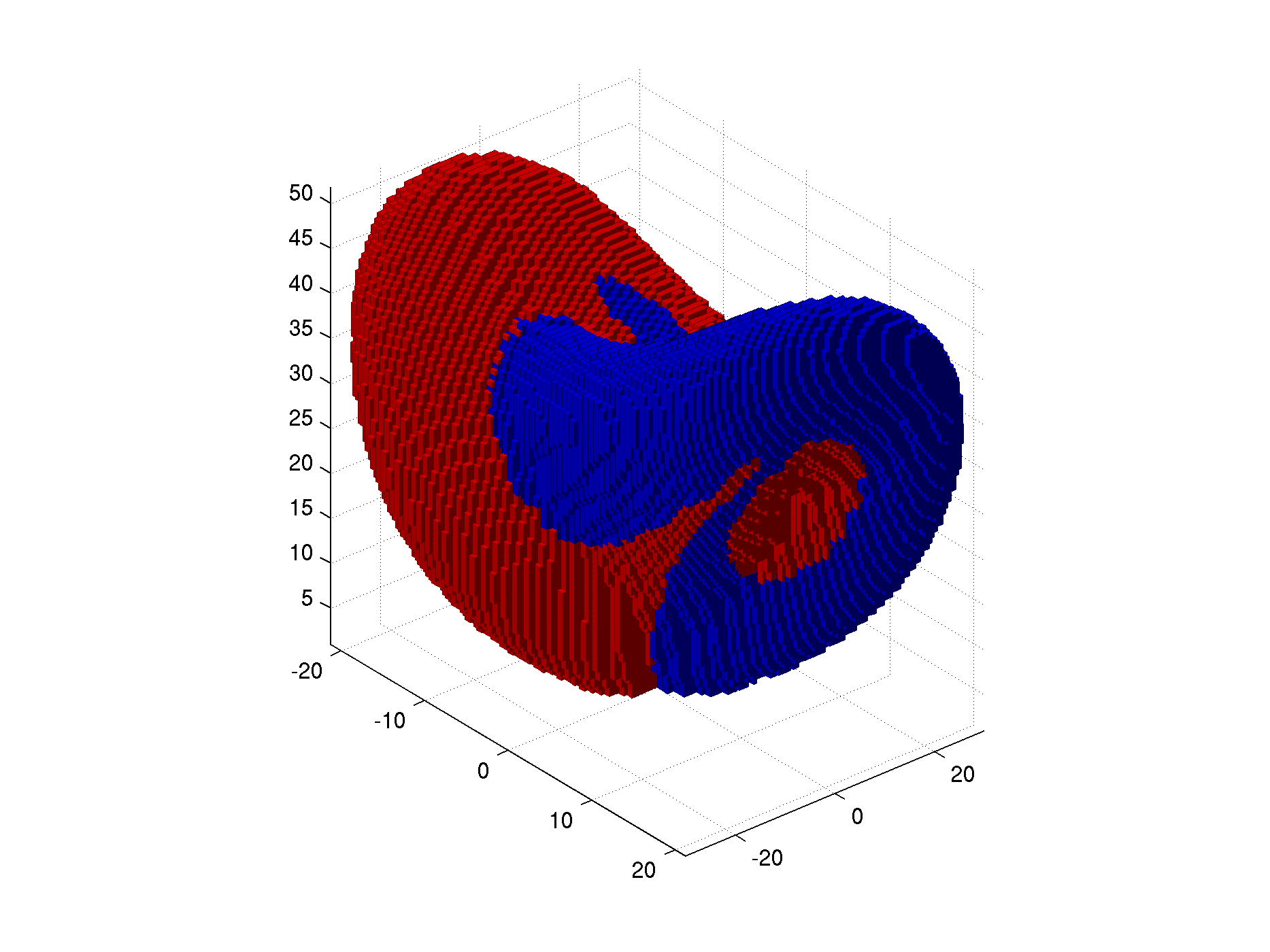}
     \includegraphics[width=0.32\textwidth]{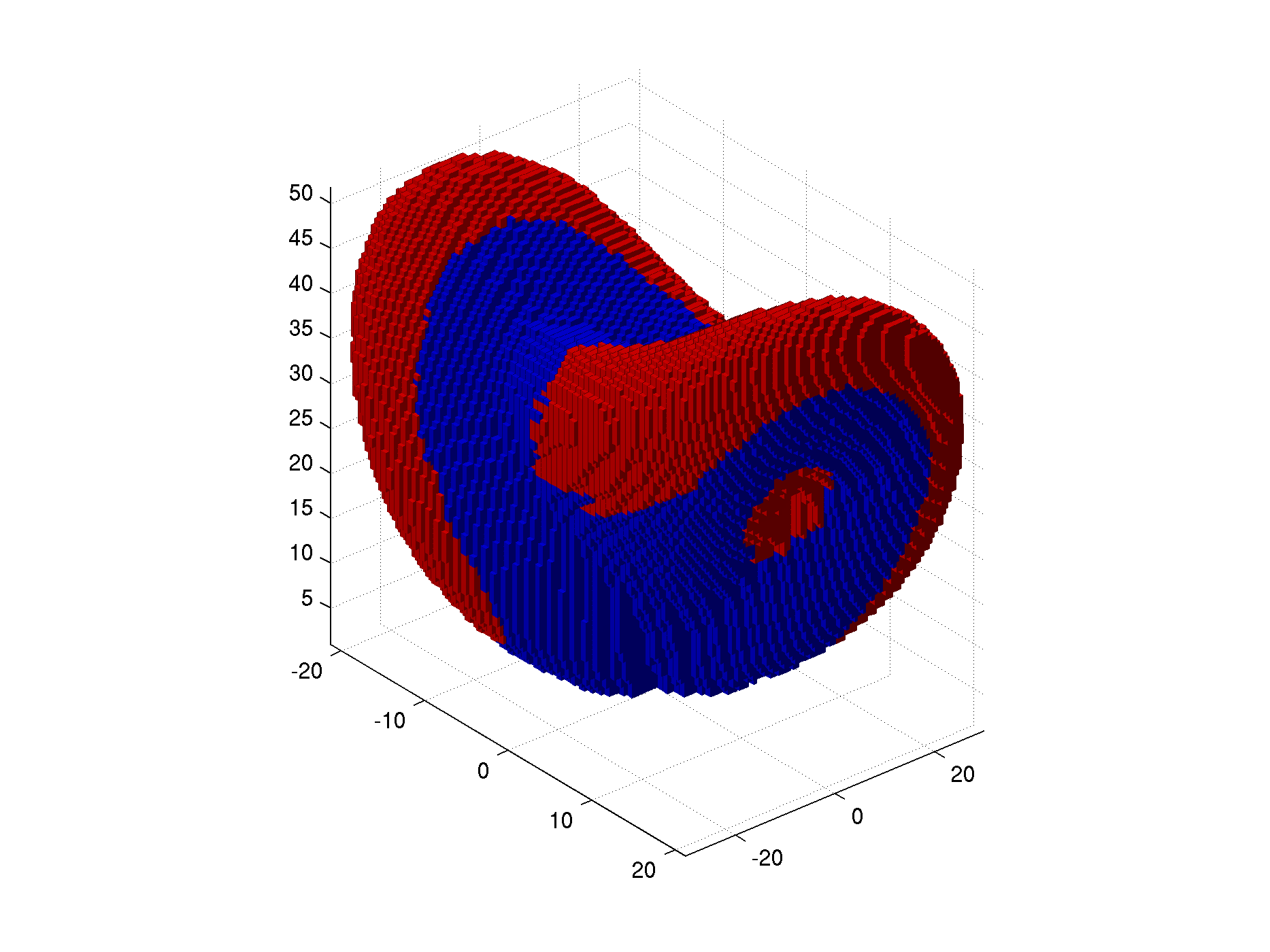}
     \caption{Ulam's method for the generator in an example with a low-dimensional attractor: The Lorenz attractor (left) and almost invariant sets: sign structure of the second eigenvector (middle) and third  eigenvector (right).}
 \label{fig:LorAttr128}
\end{figure}

\section{Conclusion}

We introduced the infinitesimal generator as a tool for studying long term behavior of flows without requiring any trajectory integration.
The (typically unit dimensional) null space of the generator is spanned by the invariant density of the flow;  the invariant density describes the asymptotic distribution of trajectories.
We showed that eigenfunctions of the generator corresponding to eigenvalues close to 0 had strong connections with almost-invariant or metastable behavior, and could be used to identify regions in the flow from which the escape rate is very low.

We proposed two new numerical approaches for the estimation of the generator; one based on Ulam's method and the other based on spectral collocation.
We tested our numerical approaches on ``full phase space'' flows in one, two, and three dimensions and found that both approaches significantly outperformed the standard Ulam approach based on the Perron-Frobenius operator.
Moreover, all operator/generator based methods were superior to standard trajectory integration for estimating the invariant density;  the identification of almost-invariant sets is essentially impossible without an operator/generator approach.
The spectral collocation approach worked particularly well in periodic domains when the vector field is infinitely smooth.

\section*{Acknowledgements}

We thank Yuri Latushkin for pointing us to useful books, and Nick Trefethen for careful proofreading and helpful questions and suggestions.

%\small
%%\bibliographystyle{abbrvnat}
\bibliographystyle{siam}
%\bibliography{bibdesk}

\end{document}